\documentclass[11pt,reqno]{article}
\usepackage{epsfig}
\usepackage{latexsym}
\usepackage{appendix}

 \newtheorem{theorem}{Theorem}[section]
 \newtheorem{lemma}[theorem]{Lemma}
 \newtheorem{corollary}[theorem]{Corollary}
 \newtheorem{proposition}[theorem]{Proposition}
 
 \newtheorem{remark}[theorem]{Remark}
 \newcommand{\qed}{\nolinebreak{\hfill  $\Box$}}
 \newenvironment{proof}{ {\bf Proof: }}{\qed \bigskip\par} 

   \newcommand{\beq}{\begin{equation}}
   \newcommand{\eeq}{\end{equation}}
   \newcommand{\bey}{\begin{eqnarray}}
   \newcommand{\eey}{\end{eqnarray}}
   \newcommand{\Bey}{\begin{eqnarray*}}
   \newcommand{\Eey}{\end{eqnarray*}}

\newcommand{\cC}{{\cal C}}
\newcommand{\cF}{{\cal F}}
\newcommand{\cJ}{{\cal J}}

\newcommand{\cM}{{\cal M}}

\newcommand{\cS}{{\cal S}}
\newcommand{\cZ}{{\cal Z}}

\newcommand{\tR}{\tilde {R}}

\newcommand{\ga}{\gamma}
\newcommand{\al}{\alpha}
\newcommand{\be}{\beta}

\newcommand{\kap}{\kappa}

\newcommand{\ve}{\varepsilon}

\newcommand{\ti}{\tilde}
\newcommand{\U}{\Upsilon}

\newcommand{\na}{\nabla}

\newcommand{\real}{{\rm I\mkern-3.7mu R}}

\newcommand{\ind}{\mbox{\rm{\small 1}\kern-4pt1}}

\newcommand{\dts}{\displaystyle}
\newcommand{\newdot}{\mbox{$\raisebox{.5ex}{\large .}\,$}}

\newcommand{\Int}{\mathrm{int}}

\newcommand{\dom}{\mathrm{dom}}

\makeatletter  
\@addtoreset{equation}{section}
\def\theequation{\arabic{section}.\arabic{equation}}
\makeatother  
\title{\vspace{-2cm} A unifying view on the irreversible investment exercise boundary \\ in a stochastic, time-inhomogeneous capacity expansion problem}
\author{Maria B. Chiarolla \thanks{e-mail address: maria.chiarolla@unisalento.it}\\
\small{\em Dipartimento di Scienze dell'Economia, Campus Ecotekne, University of Salento, 73047 Lecce, Italy}  }

\begin{document}
\maketitle

\noindent {\footnotesize {\bf  Abstract.}
This paper devises a way to apply the Bank and El Karoui Representation Theorem to find the investment boundary of a rich stochastic, continuous time capacity expansion problem with irreversible investment on the finite time interval $[0, T]$, despite the presence of a state dependent scrap value associated with the production facility at the terminal time $T$. Standard variational methods are not feasible for the proposed singular stochastic control problem but it admits some first order conditions, complicated however by an extra, non integral term  involving the scrap value function and depending on the initial capacity $y$, which are solved by devising a way to apply the Representation Theorem. Such devise, new and of interest in its own right, provides the existence of the base capacity $l^{\star}_y(t)$, a positive level which the optimal investment process is shown to become active at. As far as we know the Representation Theorem has never been applied to this extent. In the special case of deterministic coefficients, under a further assumption specific to the scrap value case, a unifying view on the curve at which it is optimal to invest emerges: the base capacity equals the investment boundary ${\hat y}(t)$ obtained by variational methods.

}

\bigskip

\bigskip



\section{Introduction}
This paper contributes to the literature on the applications of the Bank and El Karoui Representation Theorem by solving a rich stochastic, continuous time capacity expansion, with irreversible investment entering additively a time-inhomogeneous diffusion on the finite time interval $[0, T]$ and in the presence of a state dependent scrap value associated with the production facility at the terminal time $T$. Standard variational methods are not feasible to find the investment boundary of the proposed singular stochastic control model but the functional to be maximized admits a supergradient, hence the optimal control satisfies some first order conditions which are solved by devising a way to apply the Representation Theorem. Such devise is new in singular stochastic control and of interest in its own right. As far as we know the Representation Theorem has never been applied to this extent. 

The reader may find an extensive review about singular stochastic control of expanding capacity by optimally choosing the investment process in  \cite{CHi}, \cite{CH1}, among others. The classical approach is based on the study of the optimal  stopping problem naturally associated to the control problem, the so-called {\em optimal risk of not investing} (cf. for example \cite{BK}, \cite{CH1}). The optimal investment process is then usually obtained in terms of the left-continuous inverse of the optimal stopping time, and that can be done under quite general time-inhomogeneous dynamics allowing random, time dependent coefficients for the capacity process  and even, as in \cite{CHi}, multivariable production functions (cf. \cite{CHu}) often awkward to handle. However that approach nothing says about the free boundary ${\hat y} (t)$ at which it is optimal to invest. Usually, in order to obtain ${\hat y} (t)$, one must apply variational methods but for that coefficients and discount factor must be reasonably reconsidered, quite often taken constant. In some cases, under some extra conditions on the production function (e.g.  of  Cobb-Douglas type), the investment exercise boundary ${\hat y}(t)$ may be characterized by an integral equazion that often requires a priori continuity of the boundary itself. In general the investment boundary makes the strip $[0,T)\times (0, \infty)$ split into the {\em Continuation Region} $\Delta$, where it is not optimal to invest as the capital's replacement cost is strictly greater than the shadow value of installed capital,  and its complement, the {\em Stopping Region} $\Delta^c$, where it is optimal to invest instantaneously.  

This paper studies a model in which the levels of production capacity $C$, labor input $L$ (at current random wage rate $w(t)$) and operating capital $K$ (at current random interest rate $r(t)$) contribute to the firm's production and are optimally chosen in order to maximize the expected total discounted profits under the three-variable production function $R(C, L, K)$. 
The setting is borrowed by the equilibrium model in \cite{CHei}. $R(C, L, K)$ is handled by the corresponding reduced production function $\tR(C,w,r)$, given in terms of conjugate functions (cf. \cite{RTR}). The optimal capacity expansion problem is then the maximization of the expected total discounted profit plus scrap value, net of investment, 
\[
\cJ_{y}(\nu)\!=\! E\bigg\{\int_0^T \hspace{-0.2cm} e^{-\int_0^t \mu_F (u) du} \ti R(C^{y}(t;\nu), w(t), r(t))\, dt 
        + e^{- \int_0^T \mu_F (u) du}G(C^{y}(T;\nu))  - \int_{[0, T)} \hspace{-0.2cm}  e^{- \int_0^t \mu_F (u) du} \,d\nu(t) \bigg\}
\]
over the set of non-decreasing investment processes $\nu(t)$ entering the time-inhomogeneous dynamics of the capacity process $C^{y}$ starting at time zero from $y>0$,
\[
dC^{y}(t;\nu) = C^{y}(t;\nu) [-\mu_C(t)  dt + \sigma_C^{\top}(t) dW(t)] + f_C (t) d\nu(t),
\] 
where $\mu_C\geq 0, \sigma_C, f_C$  are bounded, measurable, adapted processes. 
It is evident that in such setting the classical variational approach is prohibitive. Instead, by exploiting the existence of a supergradient of the strict concave functional $\cJ_{y}(\nu)$, the optimal $\nu$ may be characterized by first order conditions which might look appealing to be solved by means of the Bank and El Karoui Representation Theorem.    

In essence the Representation Theorem (cf.  \cite{BElK}), applied to an optional processes $\{X(t), t\in [0, T ] \}$ with $X(T ) = 0$ and given a function $f(t, \xi)$ and a nonnegative optional random measure $\mu$,  provides the representation 
\beq\label{Rep in intro}
X(t) = E\Big\{\int_{(t,T]} f(s, \dts\sup_{t\leq u'<s} \xi^{\star} (u')) \,\mu(ds) \Big|\, {\cal F}_{t}\Big\}
\eeq
 in terms of a progressively measurable process $\xi^{\star}(t)$ to be determined. Its connection with some stochastic optimization problems was shown in \cite{BElK}, \cite{BF}.  A good review of applications may be found in \cite{ChF}. Application to a model of singular stochastic investment with inventory risk is more recently found in  \cite{BBess2018},  after the ealier  \cite{RS}. 
In a simplified version of the model in \cite{CHi}, the investment exercise boundary was studied in \cite{ChF} by applying  the Bank and El Karoui Representation Theorem to the first order conditions. The lack of a scrap value at the terminal time $T$ reduced the first order conditions to the conditional expectation of an integral, that is a form similar to  (\ref{Rep in intro}). It was proved the representation 
\beq\label{equation xi}
X(\tau) = E\Big\{ \int_{\tau}^T  e^{-\int_0^t \mu_F(u)\,du}C^o(t) R_C(\frac{C^o (\omega, t)}{-\dts\sup_{\tau\leq u'<t} \xi^{\star} (u')})  \,dt  \Big|\, {\cal F}_{\tau}\Big\}
\eeq
for the optional process $X(\omega,t):= e^{-\int_0^{t} \mu_F(r)\,dr}\frac{C^o(\omega, t)}{f_C(t)} \ind_{[0,T)} (t)$, 
in terms of a unique optional process $\xi^{\star} (t)$ (cf. \cite{ChF},  Lemma 4.1).  Hence the optimal investment process was obtained in terms of  $\xi^{\star}(t)$. Under the assumption of deterministic coefficients, \cite{ChF} managed to use some results in \cite{CHi} to show that the base capacity and the investment exercise boundary ${\hat y}(t)$ were the same, and therefore $l^{\star}(t)$ had to be deterministic and its integral equation could be used to characterize ${\hat y}(t)$ without any further assumption on it.

In our model the first order conditions do not have a form similar to (\ref{Rep in intro}), in fact the scrap value at the terminal time $T$ adds an extra, non integral term in the supergradient (cf. (\ref{supergradient})). Such term depends on the initial capacity $y$ through the scrap value function and it makes non trivial the application of the Representation Theorem and the existence of the unique solution of the corresponding integral equation, $l^{\star}_y(t)$, itself a function of $y$. 
We handle the presence of the scrap value  by suitably defining on $[0, \infty)$, rather than on $[0, T]$, the optional process $X(t)$, the function $f(t, \xi)$ and the optional random measure $\mu (dt)$. The method is quite involved while allowing quite general $R, G,$ and random coefficients and discount factor. 
Notice that sometimes, in much simpler contexts, the time interval is fictitiously extended to $\infty$ only in order to apply the Representation Theorem to a process $X$ that satisfies $X(\infty)=0$ but not $X(T)=0$ (see point 1. below (\ref{BElK general equation xi})), as is the case with the discounted stock price $e^{-rt} P_t$, $t \in [0, T]$, in \cite{BF}, Corollary 2.4. There, with $f(t, x)= -x$ and $\mu(dt)= r  e^{-rt} dt$, the process $ e^{-rs} P_{s\wedge {T}} $ is written as $\int_{(s, \infty]} r e^{-rt} P_{s\wedge {T}}\,  dt$ for $ s \in [0, \infty)$, and it is $0$ for $s=\infty$. 

For each initial capacity $y$, we solve the optimal capacity expansion problem by obtaining the optimal investment process in terms of the base capacity $\l^{\star}_y(t)$, a process linked to the unique solution $\xi^{\star}_y (t)$ of the integral equation given by the Representation Theorem. In fact, the base capacity is a positive level  depending on $y$ which the optimal investment process is shown to become active at. In the special case of deterministic coefficients, discount factor, conversion factor, wage rate and interest rate, we prove that a unifying view on the curve at which it is optimal to invest emerges; in fact, under a further assumption specific to the scrap value case, the base capacity equals the investment boundary ${\hat y}(t)$ obtained by variational methods, and hence it is deterministic. The advantage is that the integral equation of the base capacity may then be used to characterize ${\hat y}(t)$, whatever ``cost'' functions $w(t)$ and $r(t)$ one chooses, and without appealing to {\sc PDE} methods that could be difficult to use in complex settings.
The extra assumption allows to obtain monotonicity and strict positiveness of ${\hat y}(t)$ by means of probabilistic methods. As a byproduct of monotonicity, continuity of the optimal investment process is also obtained, despite the lack of knowledge of continuity of ${\hat y}(t)$. Such result is of interest in singular stochastic control under time-inhomogeneous dynamics.  Obviously the optimal investment process is continuous if the boundary is so, and that is the case for most models in the literature which are restricted to geometric dynamics with constant coefficients  or, more generally, limited to time-homogeneous diffusions.  

This paper is organized as follows.  
In Section~\ref{model} we set the model. 
In Section~\ref{Bank-El Karoui approach} we solve the first order conditions for optimality via the Bank and El Karoui Representation Theorem and  we obtain the optimal investment  process in terms of  the base capacity $\l^{\star}_y(t)$. 
In Section~\ref{investment exercise boundary} the special case of deterministic coefficients is considered.  The investment exercise boundary ${\hat y}(t)$ is then given by variational methods. Under a further assumption specific to the scrap value case ${\hat y}(t)$ is shown to be monotone and strictly positive on $[0,T)$. As a byproduct, the continuity of the optimal investment process is also obtained. In Section~\ref{unifying views} the above properties of the boundary are used to get a unifying view about the curve at which it is optimal to invest, in fact the base capacity is shown to coincide with the investment exercise boundary. Appendix~\ref{variational approach to free boundary} completes the paper with a brief review of the variational approach in \cite{CHi} generalized to the setting of Section~\ref{investment exercise boundary}.


\section{The Model}
\label{model}

The setting is that of the equilibrium model in \cite{CHei}. On a complete probability space $(\Omega,{\cF}, P)$ with filtration $\{ {\cF}_t: t\in [0,T] \}$ given by the usual augmentation of the filtration generated by an exogenous $n$-dimensional Brownian  motion $\{W(t): t\in [0,T] \}$, a production facility (``the firm'') with finite horizon $T$ 
 produces output at rate $R(C, L, K)$ when it has capacity $C$, employs $L$ units of labour and borrows $K$ units of operating capital at current {\em wages} $w$ and {\em interest rate} $r$, respectively. 
The capital invested in technology on the time interval $[0,t]$ to increase capacity is denoted $\nu(t)$, it is given by a process almost surely finite, left-continuous with right limits (``lcrl''), non-decreasing and adapted. The non-decreasing nature of $\nu$ expresses the irreversibility of investment.  

The three-variable production function $R(C,L, K)$ is  defined on all of $\real^3$ for convenience,  but it may take the value $-\infty$. It is finite on $\dom(R):=\{(C,L,K)^\top:R(C,L,K)>-\infty\}$. The $C$-{\em section} of $\dom(R)$ is defined by $\dom(R(C)):=\{(L,K)^\top:R(C,L,K)>-\infty\}$. We set $\na_{L,K} R(C,L,K):=(R_L(C,L,K), R_K(C,L,K))^\top$ where $R_X$ denotes the partial derivative with respect to $X$.
We denote $\real^n_{++}$  the positive orthant in $\real^n$, $\real^n_+$ its closure (i.e. the non-negative orthant). 
Let $\kap_L$ and  $\kap_K$ denote respectively the labour supply and the money, supply, and let $\kap_{w}, \kap_{r}$ be given constants. Then we make the following\\
{\bf Assumption-[R]} \vspace{-0.2cm}
\beq\label{AssR}
\left\{\begin{array}{rl}
\mathrm{(i)~}&R:\real^3\mapsto [-\infty,\infty)\mathrm{~ is~upper~semicontinuous,~concave~and~non-decreasing};\\ 
\mathrm{(ii)~}& \real^3_+\subset\dom(R),\ R(\real^3_+)\subset[0,\infty),\ R \in C(\real^3_+);  \\
\mathrm{(iii)~}& R \in C^2(\Int(\dom(R)));\\
\mathrm{(iv)~}& R \mathrm{~is~strictly~concave~and~strictly~increasing~on~}\\
                       & \hspace{0.5cm}  \Int(\dom(R))\cap(\real_+\times [0,\kap_L]\times[0,\kap_K]);\\
\mathrm{(v)~}& \lim_{C\rightarrow\infty}\inf_{(L,K)\in [0,\kap_L]\times[0,\kap_K]}R_C(C,L,K)=0.
\end{array}\right.
\eeq
{\bf Assumption-[LK]}  \vspace{-0.2cm}
\beq\label{condLK}
~~~~~ w, r: \Omega\times [0,T] \rightarrow \real \mathrm{~are~continuous~a.s.} \mathrm{~with~} 0<k_w\leq w (t)\leq \kap_{w} ~and~
0< k_r \leq r (t)\leq \kap_{r}.
\eeq
Set $A :=[0, \kap_L]\times [0, \kap_K]$ and $\ti A :=[0, \infty) \times A$. Let $R^A(C,\cdot, \cdot)$ denote $R(C,\cdot, \cdot)$ modified as $-\infty$ off $ A$. 
The manager of the firm chooses labour $L$ and operating capital $K$ in order to maximize at each moment in time the production profits $R(C, L, K) - w L - r K$, at current capacity $C$, wages $w$ and interest rate $r$.  
Therefore the maximal production profit rate is given by the ``reduced production function'',
\beq\label{R*maxprofit}
\tR(C, w, r):=\max_{(L,K)\in A} ~[R(C, L, K) - w L- r K]. 
\eeq
Notice that, for fixed $C$,  $\tR(C,w, r)$ is the negative of the concave conjugate of $R^A(C,\cdot, \cdot)$ (cf. \cite{RTR}), hence it is convex in $w$ and in  $r$, and  strictly concave in $C$  (by a generalization of Proposition~5.1 in the Appendix of \cite{CHi}). 
As in \cite{CHi}, Section~2 (with $L$ replaced by $(L,K)$) a unique solution exists and it is denoted by 
$$(L^C(w,r), K^C(w,r))^\top:=I^{R^A (C,\newdot,\newdot)}(w,r)$$ with $\top$ denoting ``transpose'' and 
$I^{R^A(C,\newdot,\newdot)}$  being an extension of the inverse of $\na_{L,K} R^A (C,\newdot,\newdot)$  (cf. \cite{CHu}, Proposition~3.2).
\begin{remark}\label{growth condition}\rm
The growth condition 
$\sup_{C\geq 0}\;\max_{ [0,\kap_{w}]\times [0,\kap_{r}] }\{\ti R(C,w,r) -\ve C\}
=\kap_{\ve} <\infty$ 
holds with $\kap_{\ve}$ depending on $\kap_L$, $\kap_{w}$, $\kap_K$, $\kap_{r}$ and $\ve$. It follows by \cite{CHu}, Proposition~3.3, thanks to Assumption-[R] (see also \cite{CHi}). It is needed to obtain Proposition \ref{cJprop} below. 
{\qed \par}\end{remark}
The capacity process $C^{y}(t;\nu)$ starting at time zero from level $y>0$ under the control $\nu$ is given by
\beq
\label{C.eq starting at 0}
\left\{
\matrix{
dC^{y}(t;\nu) = C^{y}(t;\nu) [-\mu_C(t)  dt + \sigma_C^{\top}(t) dW(t)] + f_C (t) d\nu(t),              
\qquad t\in (0, T], \hfill\cr
C^{y}(0;\nu) =y> 0, \hfill\cr
}
\right.
\eeq
where $f_C$ is a conversion factor as one unit of investment is converted into $f_C$ units of production capacity,
and $\nu\in {\cal S}$ with 
\[
{\cal S} := \{ \nu: [0,T] \rightarrow \real: \mbox{\rm ~non-decreasing, lcrl, adapted  process with ~} \nu(0)=0 \mbox{\rm ~ a.s.} \}, 
\]
a convex set. For the coefficients it holds the following\\
{\bf Assumption-[C]} \vspace{-0.2cm}
\beq\label{AssC}
\left\{\begin{array}{rl}
\mathrm{(i)~}&\mu_C\geq 0, \sigma_C, f_C \mathrm{~are~bounded,~measurable,~ adapted~processes~on~} [0,T];\\ 
\mathrm{(ii)~}& f_C \mathrm{~is~continuous~with~} 0<k_f\leq f_C\leq \kap_f. 
\end{array}\right.
\eeq
In the absence of investment, the decay of a unit of initial capital is 
\beq
\label{C^o.eq}
C^o(t):=C^{1}(t;0) = e^{-\int_0^t \mu_C(u)du} \cM_0 (t),
\eeq
with  
\beq
\label{martingale}
\cM_s (t) 
= e^{[\int_s^t\sigma_C^{\top}(u)\, dW(u) - \frac{1}{2} \int_s^t \| \sigma_C(u)\|^2 du]}, \qquad t \in [s,T], 
\eeq
exponential martingale defined for each fixed $s \in [0,T]$ and such that $E\{[\cM_s (t)]^p\} < \infty$ for any $p$. 
Then the solution of (\ref{C.eq starting at 0}) is written as
\beq
\label{C solution}
C^{y}(t;\nu) =C^o(t)\Big[y + \int_{[0,t)}\frac{f_C(u)}{C^o(u)}\,d\nu(u)\Big] =C^o(t)[y + {\overline\nu}(t)], 
\eeq
with ${\overline\nu}(t) := \int_{[0,t)}\frac{f_C(u)}{C^o(u)}\,d\nu(u)$. 


The firm has a finite planning horizon $T$ and at that time  its value is $G(C^{y}(T;\nu))$, the so-called {\em scrap value} of the firm. For $G$ it holds the following \\
{\bf Assumption-[G]} \vspace{-0.2cm}
\beq\label{condG}
\left\{\begin{array}{rl}
\mathrm{(i)~}&G:\real_+\mapsto\real_+ \mathrm{~is~concave,~non-decreasing,~continuously~ differentiable};\\
\mathrm{(ii)~}&
\lim_{C\rightarrow\infty}G'(C)=0,\qquad
  G'(0)f_C(T)\leq 1\mathrm{\rm ~a.s.}.\\
\end{array}\right.
\eeq
An alternative to the limit condition in (ii) above is:
$G(C)\leq a_o+b_o C,\ b_o\kap_f<1,\ a_o,b_o\geq 0.$

Over the finite planning horizon the firm chooses the investment process $\nu \in \cS$
in order to maximize the expected total discounted profit plus scrap value, net of investment, 
\begin{eqnarray}
\label{J}
& & 
\cJ_{y}(\nu)= E\bigg\{\int_0^T  e^{-\int_0^t \mu_F (u) du} \ti R(C^{y}(t;\nu), w(t), r(t))\, dt 
        + e^{- \int_0^T \mu_F (u) du}G(C^{y}(T;\nu)) \\
&  & \hspace{2.1cm}  - \int_{[0, T)}  e^{- \int_0^t \mu_F (u) du} \,d\nu(t) \bigg\}. \nonumber
\end{eqnarray}
Here  the firm discount rate $\mu_F$ 
is assumed to be uniformly bounded in $(t,\omega)$, nonnegative, measurable, adapted with $\bar\mu:=\mu_C+\mu_F\geq\ve_o>0 $ a.s. 
Note that  $(t,\omega)\mapsto \tR(C^{y}(t;\nu), w(t), r(t))$ is measurable due to the continuity of $\tR(C, w, r)$. 
The firm's optimal capacity expansion problem is  then
\beq\label{cap prob from 0}
V(y) := \dts{\max_{\nu \in {\cal S}}}\ \cJ_{y}(\nu). 
\eeq 
Problem (\ref{cap prob from 0}) is a modification of the ``social planning problem" discussed by Baldursson and Karatzas \cite{BK}. Its solution may be found in terms of the left-continuous inverse of the optimal stopping time  of the associated optimal stopping problem known as ``the optimal risk of not investing'' as in \cite{CHi}, Section 3.  In fact \cite{CHi} generalizes the results in \cite{CH1}.
Here the unboundness of the reduced production function $\ti R$ is handled by the following estimates which allow to use some results in \cite{BK}. Its proof is as \cite{CHi}, Proposition~2.1.
\begin{proposition}\label{cJprop}
There exists a constant $K_\cJ$  depending on $T, \kap_L, \kap_w, \kap_K, \kap_r, \kap_f, k_f$ only such that
\[ 
\begin{array}{ll}
(a)& \cJ_{y}(\nu)\leq K_\cJ(1+y)\,\mbox{\rm ~on~} \cS,\\
& \\
(b)& {\dts E\bigg\{\int_{[0,T)} e^{-\int_0^t\mu_F(u)\,du}\,d\nu(t)\bigg\}}\leq 2K_\cJ(1+y) \mbox{\rm ~~~if~~}\cJ_{y}(\nu)\geq 0.
\end{array}
\]
\end{proposition}
Notice that $\tR$ strictly concave in $C$, $G$ concave, and $C$ affine in $\nu$ imply that $\cJ_{y}$ is concave on $\cS$; in fact it is strictly concave since  $0<k_f\leq f_C$. 
The strict concavity of $\cJ_{y}(\nu)$ has a double implication. First, it implies that the solution of the optimal capacity expansion problem (\ref{cap prob from 0}) is unique, if it exists. If we denote it by $\hat\nu$, then the corresponding unique solution of (\ref{R*maxprofit}) is given by
\beq\label{hKL}
({\hat L}(t),{\hat K}(t))^\top := (L^{C^{y}(t;\hat\nu)}(w(t), r(t)), K^{C^{y}(t;\hat\nu)}(w(t), r(t)))\\
= I^{R^A (C^{y}(t;\hat\nu),\newdot,\newdot)}(w(t), r(t)). \nonumber
\eeq
Secondly, the functional $\cJ_{y}$ admits the supergradient (\ref{supergradient}) in the next section.
\begin{remark}\label{Appendix of [5]}\rm
Notice that there is a simpler way to calculate $\tR_C(y C^s(u), w(u), r(u))$. In fact, a generalization of Proposition 5.1 in the Appendix of \cite{CHi} (based on the results in \cite{CHu}) implies that
\beq\label{formula Appendix of[8]}
\tR_C(C, w, r) = R_C (C, I^{R^A (C,\newdot,\newdot)}(w, r)),
\eeq
for $w$ and $r$ fixed, since $[\nabla_{L,K} R(C, I^{R^A (C,\newdot,\newdot)}(w, r)) - (w, r)^{\top}]^{\top} \frac{\partial}{\partial C}I^{R^A (C,\newdot,\newdot)}(w, r) = 0$. For example, if $R$ is of the Cobb-Douglas type with zero shift, i.e. $R(C,L,K)=\frac{1}{\al\be\ga}C^\al L^\be K^\ga$ with $\al,\be,\ga>0$ and $\al+\be+\ga<1$, then 
\[
\tR_C(C, w(t), r(t))=\left[\frac{1}{\beta \gamma}\left(\frac{\beta}{\alpha \, w(t)}\right)^{\beta}\left(\frac{\gamma}{\alpha \, r(t)}\right)^{\gamma} C^{\alpha+\beta+\gamma -1}\right]^{\frac{1}{1-\beta-\gamma}}.
\]
{\qed \par}\end{remark}

\section{Solving by the Bank-El Karoui Representation Theorem approach}
\label{Bank-El Karoui approach}
To solve the optimal capacity expansion (\ref{cap prob from 0}) by finding the curve at which it is optimal to invest is not an easy task. The underline dynamics has random time-depending coefficients and the production function is also a function of the stochastic processes $w (t), r(t)$. So variational methods are precluded and we must take a different approach. 

We look for some first order conditions. The strict concavity of $\cJ_{y}(\nu)$ implies the existence of the supergradient,
\bey\label{supergradient}
& & 
\nabla_{\nu} \cJ_{y} (\nu)(\tau) = \frac{f_C(\tau)}{C^o(\tau)} E\bigg\{ \int_{\tau}^T  e^{-\int_0^t \mu_F(u)\,du}C^o(t) \tR_C(C^{y}(t;\nu), w(t), r(t))  \,dt \\
&  & \hspace{4cm} +  e^{-\int_0^T \mu_F(u)\,du} C^o(T) \, G'(\frac{C^{y}(T;\nu)}{\ind_{[0,T)} (\tau)})  \,\bigg|\,\cF_{\tau}\bigg\} 
- e^{-\int_0^{\tau} \mu_F(u)\,du} \ind_{[0,T)} (\tau) \nonumber
\eey
for all $\{{\cal F}_{t}\}_{t\in [0,T]}$-stopping times $\tau : \Omega \rightarrow [0,T]$ and $\omega\in \Omega$ such that $\tau(\omega) <T$ (recall that $G'(\infty)=0$). 
Then optimality is characterized through first-order conditions.  

Denote by $\U[0,T]$ the set of all $\{{\cal F}_{t}\}_{t\in [0,T]}$-stopping times taking values in $[0,T]$.  A generalization of the proof of \cite{ChF}, Theorem 3.2, that takes into account the presence of the scrap value at the terminal time $T$ in the present model, that is the presence of the non integral term involving $G'$  and the initial capacity $y$ in the conditional expectation appearing in the supergradient,  gives 
\begin{proposition}\label{First Order Cond}
The control $\hat\nu \in \cS$ is optimal for the capacity problem {\rm (\ref{cap prob from 0})} if and only if it satisfies the first-order conditions 
\beq\label{FOC}
\left\{
\matrix{
\nabla_{\nu} \cJ_{y} (\hat\nu)(\tau) \leq 0 \quad\mbox{a.s.~~~for~} \tau\in \U[0,T], \vspace{0.2cm}\hfill\cr
E\bigg\{\dts\int_{[0,T)}\nabla_{\nu} \cJ_{y} (\hat\nu)(t) \, d \hat\nu (t)\bigg\} = 0. \hfill\cr
}
\right.
\eeq
\end{proposition}
In particular,  
 (\ref{FOC})$_2$ says that $\nabla_{\nu} \cJ_{y} (\hat\nu)(t)$ is zero at times $t$ of strict increase for 
$\hat\nu$, whereas by (\ref{FOC})$_1$ for any $\tau\in \U[0,T]$,
$E\Big\{\int_{\tau}^T  e^{-\int_0^t \mu_F(u)\,du}C^o(t) \tR_C(C^{y}(t;\hat\nu), w(t), r(t))  dt 
+  e^{-\int_0^T \mu_F(u)\,du} C^o(T) \, G'(\frac{C^{y}(T;\hat\nu)}{\ind_{[0,T)} (\tau)})  \Big|\,\cF_{\tau}\Big\} $ 
is lesser or equal to $e^{-\int_0^{\tau} \mu_F(u)\,du}  \frac{C^o(\tau)}{f_C(\tau)} \ind_{[0,T)} (\tau)$ a.s.

We aim to solve (\ref{FOC}) by means of the Bank-El Karoui Representation Theorem (cf. \cite{BElK}, Theorem 3). This celebrated Theorem, for $T$ not necessarily finite, guarantees the representation 
\beq\label{BElK general equation xi}
X(\tau) = E\Big\{\int_{(\tau, T]} f(t, \dts\sup_{\tau\leq u'<t} \xi^{\star} (u')) \,\mu(dt) \Big|\, {\cal F}_{\tau}\Big\},
\quad \tau \in \U[0,T],
\eeq
in terms of a unique optional process $\xi^{\star}$,  given
\vspace{-0.5cm}\\
\begin{enumerate}
\item $X(\omega,t) :\Omega \times [0,T] \rightarrow \real$ an optional process  of class (D) (i.e. $\sup_{\tau\in \U[0,T]} E\{X(\tau)\} <+\infty$), lower-semicontinuous in expectation (i.e. for any sequence of stopping times $\tau_n \nearrow \tau$ a.s. for some $\tau \in \U[0,T]$, it holds $\liminf_n E\{X(\tau_n)\} \geq E\{X(\tau)\}$) and such that $X(T) =0$, \vspace{-0.5cm}\\
\item $\mu(\omega, dt)$ a  nonnegative optional random Borel measure, \vspace{-0.5cm}\\
\item $f(\omega, t, x) : \Omega\times [0,T]\times \real \rightarrow \real$ a random field such that \vspace{-0.2cm}
\begin{description}
\item[(i)]  for any $x \in \real$, the stochastic process $(\omega, t) \rightarrow f(\omega, t, x)$ is progressively measurable and belonging to $L^1(P(d\omega) \otimes \mu(\omega, dt))$, \vspace{-0.5cm}\\
\item[(ii)]  for any $(\omega, t) \in \Omega \times [0,T]$, the mapping $x \rightarrow f(\omega, t, x)$ is continuous and strictly decreasing from $+\infty$ to $-\infty$.\vspace{-0.5cm}\\
\end{description}
\end{enumerate}
Moreover, the unique optional process $\xi^{\star}$ is found in terms of the optimal stopping problem, for $\xi <0$ and $t\in [0,T]$,  
\beq\label{Gamma standard}
\Gamma^{\xi} (t) :=\mbox{ess}\dts\inf_{\hspace{-0.5cm} t\leq\tau\leq T} E\Big\{\dts\int_t^{\tau} f(u, \xi) \, \mu(du) + X(\tau) \Big|\, {\cal F}_t\Big\}. 
\eeq
In fact, an application of  \cite{BElK}, Lemma 4.12,  implies  the continuity of the mapping $\xi \rightarrow \Gamma^{\xi} (\omega, t)$ for any fixed $(\omega, t)\in \Omega \times [0,T]$,  its non-increasing property from $\Gamma^{-\infty} (t) := \lim_{\xi \downarrow -\infty}\Gamma^{\xi} (t) = X(t)$, and the optimality of the stopping time 
\beq\label{equation tau xi}
\tau^{\xi}(t) := \inf\{u\in [t, T) : \Gamma^{\xi} (u) = X(u)\}\wedge T, \qquad t\in [0, T),
\eeq
for $\Gamma^{\xi}(t)$,  which may be taken right-continuous in $t$. 
Then the Bank-El Karoui Representation Theorem  (cf. \cite{BElK}, equation (23) and Lemma 4.13) showes that the solution $\xi^{\star}$ is the optional process 
\beq\label{optional solution}
\xi^{\star}(t) := \sup\{\xi <0: \Gamma^{\xi} (t) = X(t)\}, \qquad t\in [s, T),
\eeq
which is also uniquely determined up to optional sections if shown to be upper right-continuous \`{a} la Bourbaki (cf. \cite{BElK}, Theorem 1), 
 i.e. if  $\xi^{\star}(s) = \dts\lim_{\ve\downarrow 0} \dts\sup_{u\in[s, (s+ \ve)\wedge T]} \xi^{\star}(u)$ 
 (such limit is greater than or equal to what is commonly called limit superior, in fact it is its upper envelope).  

Notice that in the no scrap value model  of \cite{ChF} the representation (\ref{equation xi}) was obtained thanks also to the 
Inada condition holding for the production function $R$ that  allowed to satisfy point $3_{\rm (ii)}$ above. 
There $X$ and $\xi^{\star}$ were independent of the initial capacity $y$ as the first order conditions did not contain the $G'$ term. Also $\Gamma^{\xi}$ was independent of $y$, and  it was non-increasing from $\Gamma^{-\infty} (t) := \lim_{\xi \downarrow -\infty}\Gamma^{\xi} (t) = X(t)$ again by the Inada condition. In the present model too we shall need to assume some sort of  Inada condition. 
We make the further \\
{\bf Assumption-[I]} \vspace{-0.2cm}
\beq\label{Inada}
\left\{\matrix{ 
\lim_{C\rightarrow 0} \tR_C(C,w,r)=+ \infty \quad \mbox{for any~} (w,r)\in (0,\kap_{w}]\times(0,\kap_{r}],\hfill\cr   
G':\real_+\mapsto\real_+  \mbox{~~strictly  decreasing}. \hfill\cr
}
\right.
\eeq

To handle the scrap value at terminal time $T$ we shall suitably define the processes $X, \mu, f$  on $[0, \infty)$ in order to satisfy all the requirements while trying to recover the non integral term in (\ref{supergradient}) by integrating over the extra time interval $[T, \infty)$. Such devise is new in singular stochastic control and of interest in its own right. As far as we know the Representation Theorem has never been applied to this extent. So we set
\beq\label{defining X}
\left\{\matrix{ 
 X^{y}(\omega, t) :=
\left\{\matrix{ 
e^{-\int_0^t \mu_F (\omega, u)du} \dts\frac{C^o (\omega, t)}{f_C(\omega, t)}, 
\hspace{5.3cm} t\in [0, T),\hfill\cr   
 e^{- \int_0^T \mu_F (\omega, u)du} e^{- \mu_F (\omega, T) (t-T) } C^o (\omega, T)\, G'(y C^o (\omega, T)),   \hspace{0.35cm} t\in [T, \infty);\vspace{0.2cm}\hfill\cr  
}
\right.
\hfill\cr  
\mu(\omega, dt) :=\left\{\matrix{ 
 e^{-\int_0^t \mu_F (\omega, u)du} C^o (\omega, t)\, dt, \hspace{3.65cm} \mbox{~on~} [0, T),\hfill\cr   
 e^{- \int_0^T \mu_F (\omega, u)du} e^{-  \mu_F (\omega, T)(t-T)} C^o (\omega, T)\, dt, \hspace{1.1cm} \mbox{~on~} [T, \infty);\vspace{0.2cm}\hfill\cr  
}
\right.
\hfill\cr  
f(\omega, t, x) := \left\{\matrix{ 
\tR_C ( \frac{C^o (\omega, t)}{-x},~w(t), r(t)), \hspace{3.9cm} x<0, ~t \in [0, T), \hfill\cr   
\mu_F (\omega, T)\, G'(\frac{C^o (\omega, T)}{-x}),  \hspace{4.2cm} x<0, ~t\in [T, \infty), \hfill\cr   
~ - x, \hspace{6.7cm} ~x\geq 0. \hfill\cr   
}
\right.
\hfill\cr  
}
\right.
\eeq
\begin{proposition} \label{representation up to infty}
Under Assumption-{\rm [I]}, for $y>0$  and $X, \mu, f$ given by (\ref{defining X}), it holds the representation
\beq\label{equation xi up to infty}
X^{y}(\tau) = E\Big\{\int_{\tau}^{\infty} f(t, \dts\sup_{\tau\leq u'<t} \xi^{\star}_y (u')) \,\mu(dt) \Big|\, {\cal F}_{\tau}\Big\}, \quad \tau \in \U[0,T],
\eeq
for a unique optional, right-continuous, strictly negative process $\xi^{*}_y (t)$.
\end{proposition}
\begin{proof}
Set (cf.(\ref{Gamma standard}))
\beq\label{equation Gamma-y up to infty}
\Gamma^{y; \xi} (t) :=\mbox{ess}\dts\inf_{\hspace{-0.5cm} t\leq\tau <\infty} E\Big\{\dts\int_t^{\tau} f(u, \xi) \mu(du) + X^{y}(\tau) \Big|\, {\cal F}_t\Big\} \qquad \mbox{for ~} \xi <0,\, t\in [0, \infty).
\eeq
By (\ref{defining X}) it is clear that $X^{y}(\infty) =0$, and all the requirements about the process $X$, the random measure $\mu(\omega, dt)$ and the random field $f(\omega, t, x)$ are satisfied. In fact for $t\geq T$ our $f(\omega, t, x)$ is strictly decreasing from $\mu_F (T)\, G'(0)$ to $-\infty$, rather than from $-\infty$ to $-\infty$,  but this is all that is needed to obtain $\Gamma^{y;-\infty} (t) = X^{y} (t)$. Then the solution $\xi^{*}_y$ exists and it is unique. Its upper right-continuity and strict negativity on $[0,T)$ follow from arguments as in \cite{ChF},  Lemma 4.1, based on the definition of $f(\omega, t, x)$ and the right-continuity of $\Gamma^{y; \xi} (t)$.
\end{proof}
It follows that  
\beq\label{optimal stopping time for us}
\tau^{y;\xi}(t) := \inf\{u\in [t, \infty) : \Gamma^{y;\xi} (u) = X(u)\}
\eeq
is optimal for $\Gamma^{y;\xi}(t)$, for $t\geq 0$ (cf.  (\ref{equation tau xi})). The next Lemma provides the explicit calculation of $\Gamma^{y;\xi} (t)$.
\begin{lemma} \label{calculating Gamma}
Under Assumption-{\rm [I]}, for $y>0$, $\xi<0$, and $t\in [0,T)$, 
\bey\label{Gamma y,xi}
& & \hspace{-1.8cm}  \Gamma^{y;\xi} (t) =  \mbox{ess}\dts\inf_{\hspace{-0.6cm} t\leq\tau\leq T} E\bigg\{\int_t^{\tau} e^{- \int_0^u \mu_F (r) dr}  C^o (u) \tR_C(\frac{C^o(u)}{-\xi},  w(u),  r(u))\,du \nonumber\\
& &\hspace{1.2cm} + e^{-\int_0^{\tau} \mu_F (r) dr} \frac{C^o(\tau)}{f_C(\tau)}\ind_{\{\tau<T\}}
+ e^{- \int_0^T \mu_F (r) dr}  C^o (T) G'(y C^o(T)) \ind_{\{\tau= T\}}~\Big| {\cal F}_t \bigg\}.
\eey
Whereas  for $t \in [T, \infty)$, 
\bey\label{Gamma y,xi from T on}
 &  & \hspace{-1.8cm}   \Gamma^{y;\xi} (t) =\mbox{ess}\dts\inf_{\hspace{-0.6cm} t\leq\tau <\infty} E\Big\{
 e^{- \int_0^T \mu_F (r) dr} C^o (T) \nonumber\\
  & & \hspace{1.8cm} \times \Big[\Big(1-  e^{- \mu_F (T)(\tau-T)}\Big) G'(\frac{C^o(T)}{-\xi})
 +  e^{- \mu_F (T)(\tau-T)}G'(y C^o (T)) \Big]\,\Big| {\cal F}_t \Big\}.
 \eey
\end{lemma}
\begin{proof}
For $t \in [T, \infty)$ an explicit calculation of $\Gamma^{y;\xi} (T)$ gives (\ref{Gamma y,xi from T on}). 
For $t\in [0,T)$ one has
\Bey
& & \hspace{-1.3cm}  \Gamma^{y;\xi} (t) :=\mbox{ess}\dts\inf_{\hspace{-0.6cm} t\leq\tau <\infty} E\bigg\{\!\!\int_t^{\tau \wedge T}\!\!\! e^{- \int_0^u \mu_F (r) dr} 
C^o (u) \tR_C\Big(\frac{C^o(u)}{-\xi},  w(u), \! r(u)\!\Big)\,du \!+\! e^{-\int_0^{\tau} \mu_F (r) dr} \frac{C^o(\tau)}{f_C(\tau)} 
\ind_{\{\tau<T\}}\\
& &  
\hspace{2cm} +  e^{- \int_0^T \mu_F (r) dr}  C^o (T) \Big([1- e^{- \mu_F (T) (\tau -T)}] G'(\frac{C^o(T)}{-\xi}) \ind_{\{\tau > T\}} \\
& & 
\hspace{2cm} + e^{- \mu_F (T) (\tau-T) }G'(y C^o (T)) \ind_{\{\tau \geq T\}}\Big)~\Big| {\cal F}_t \bigg\}\\
& & \geq 
\mbox{ess}\dts\inf_{\hspace{-0.6cm} t\leq\tau <\infty} E\bigg\{\int_t^{\tau\wedge T}
\hspace{-0.1cm} e^{- \int_0^u \mu_F (r) dr} 
C^o (u) \tR_C\Big(\frac{C^o(u)}{-\xi},  w(u),  r(u)\Big)\,du 
+ e^{-\int_0^{\tau} \mu_F (r) dr} \frac{C^o(\tau)}{f_C(\tau)} \ind_{\{\tau<T\}}\\
& & \hspace{2cm} +e^{- \int_0^T \mu_F (r) dr} C^o (T) G'(y C^o (T))\ind_{\{\tau = T\}}~\Big| {\cal F}_t \bigg\};
\Eey
due to  $\mu_F >0$ and $G'>0$. On the other hand, by restricting the set of stopping times, one has
\Bey
& & \hspace{-0.5cm}  \Gamma^{y;\xi} (t)\leq  \mbox{ess}\dts\inf_{\hspace{-0.6cm} t\leq\tau\leq T} E\bigg\{\int_t^{\tau} e^{- \int_0^u \mu_F (r) dr}  C^o (u) \tR_C\Big(\frac{C^o(u)}{-\xi},  w(u),  r(u)\Big)\,du \nonumber\\
& &\hspace{2.5cm} + e^{-\int_0^{\tau} \mu_F (r) dr} \frac{C^o(\tau)}{f_C(\tau)}\ind_{\{\tau<T\}}
+ e^{- \int_0^T \mu_F (r) dr}  C^o (T) G'(y C^o(T)) \ind_{\{\tau= T\}}~\Big| {\cal F}_t \bigg\}
\Eey
and the result follows.
\end{proof}
Since the terminal time of our capacity problem is $T$, we need to know what values  $\xi^{\star}_y $ takes on $[T, \infty)$. We have
\begin{lemma}\label{calculating xi from T on}
Under Assumption-{\rm [I]}, for $y>0$, 
\beq\label{xi-star after T}
\xi^{\star}_y (t) =-\frac{1}{y} \quad \mbox{~for all } t \geq T.
\eeq
\end{lemma}
\begin{proof}
Recall that $\xi^{\star}_y(t) :=  \sup\{\xi<0: \Gamma^{y;\xi} (t) = X^{y} (t)\}$ (cf. (\ref{optional solution})).

By (\ref{Gamma y,xi from T on}) for $t=T$, if $ -\frac{1}{y}<\xi <0$ then  $\Gamma^{y;\xi} (T)=e^{- \int_0^T \mu_F (r) dr} C^o (T)G'(\frac{C^o(T)}{-\xi}) < X^{y} (T)$   with the infimum attained at $\tau^{y;\xi}(T)=\infty$. Whereas for $\xi< -\frac{1}{y}$ the infimum is attained at $\tau^{y;\xi}(T)=T$ and  $\Gamma^{y;\xi} (T)=e^{- \int_0^T \mu_F (r) dr} C^o (T)G'(y C^o(T))$ is independent of $\xi$  (so $\Gamma^{y;-\infty} (T) = X^{y} (T)$ is satisfied); that same value of $\Gamma^{y;\xi} (T)$
is found for $\xi = -\frac{1}{y}$ since $\Gamma^{y; -\frac{1}{y}}(T)$ is the infimum of a constant argument, independent of $\tau$. 
Therefore $\xi^{\star}_y(T) = -\frac{1}{y}$.

Similarly, for $t>T$  (cf. (\ref{Gamma y,xi from T on}))
\Bey
& & \hspace{-0.4cm}\Gamma^{y;\xi} (t) = \mbox{ess}\dts\inf_{\hspace{-0.8cm} t\leq\tau <\infty} E\Big\{
 e^{- \int_0^T \mu_F (r) dr} C^o (T) \\
 & & \hspace{2.8cm} \times \Big[G'(\frac{C^o(T)}{-\xi})+ e^{- \mu_F (T)(\tau-t)} \Big(e^{- \mu_F (T)(t-T)}G'(y C^o (T))   -  G'(\frac{C^o(T)}{-\xi})\Big)\Big]~\Big| {\cal F}_t \Big\}.
\Eey
Therefore, if $e^{- \mu_F (T)(t-T)}G'(y C^o (T))  >  G'(\frac{C^o(T)}{-\xi})$, then the infimum is attained at 
$\tau^{y;\xi}(t)=\infty$,  $\Gamma^{y;\xi} (t)=e^{- \int_0^T \mu_F (r) dr} C^o (T)G'(\frac{C^o(T)}{-\xi}) < X^{y} (t)$ and $\xi > -\frac{1}{y}$, since $G'(y C^o (T))  >  G'(\frac{C^o(T)}{-\xi})$. Instead,  $G'(y C^o (T)) -  G'(\frac{C^o(T)}{-\xi})\leq 0$ for $\xi \leq -\frac{1}{y}$, so that 
$G'(y C^o (T))  -  e^{\mu_F (T)(t-T)} G'(\frac{C^o(T)}{-\xi}) < 0$. Hence the infimum is attained at
$\tau^{y;\xi}(t)=t$ and $\Gamma^{y;\xi} (t)=e^{- \int_0^T \mu_F (r) dr} C^o (T)e^{- \mu_F (T)(t-T)}G'(y C^o (T))  = X^{y} (t)$ is independent of $\xi$ (again  $\Gamma^{y;-\infty} (t) = X^{y} (t)$), and (\ref{xi-star after T}) follows. 
\end{proof}
Notice that  for $t\in [0,T)$ and $\xi \leq -\frac{1}{y}$  the optimal stopping time 
$\tau^{y;\xi}(t) $ of $\Gamma^{y;\xi} (t)$ (cf. (\ref{optimal stopping time for us})) reduces to 
\beq\label{best time up to T}
\tau^{y;\xi}(t) = \inf\Big\{u\in [t, T) : \Gamma^{y;\xi} (u) =e^{-\int_0^u \mu_F (r) dr} \frac{C^o(u)}{f_C(u)} \Big\}\wedge T.
\eeq
Using Lemma \ref{calculating xi from T on} we obtain

\begin{proposition} \label{application B-ElK}
Under Assumption-{\rm [I]}, for $y>0$ there exists a unique optional, upper right-continuous, negative process $\xi^{*}_y (t)$ that, for all $\tau \in \U[0,T]$,  solves
the representation problem
\bey\label{existence of xi star}
& & \hspace{-1.5cm} e^{-\int_0^{\tau} \mu_F (r)\,dr} \frac{C^o(\tau)}{f_C(\tau)}\ind_{[0,T)} (\tau) 
= E\bigg\{\int_{\tau}^{T} e^{-\int_0^t \mu_F (r)\,dr}  C^o (t) \tR_C \Big( \frac{C^o (t)}{- \dts\sup_{\tau\leq u'<t}
 \xi^{\star}_y (u')}, w(t), r(t) \Big) dt \\
& & \hspace{3.6cm}  
+ e^{- \int_0^T \mu_F (r)\,dr}  C^o (T) G' \Big(\frac{C^o (T)}{\Big[({- \dts\sup_{\tau\leq u' < T}} \xi^{\star }_y (u'))  \wedge \frac{1}{y}\Big]\ind_{[0,T)}(\tau)} \Big) \Big| {\cal F}_{\tau} \bigg\}. \nonumber 
\eey
\end{proposition}
\begin{proof}
It suffices to apply the Bank-El Karoui Representation Theorem to the optional process $X^{y}(t)$,  the nonnegative optional random Borel measure $\mu(\omega, dt)$, and the random field $f(\omega, t, x)$ defined in (\ref{defining X}). 
From (\ref{equation xi up to infty}) with $\tau=T$ we have
\[
X^{y}(T) =  e^{- \int_0^T \mu_F (r)\,dr} C^o (T) E\bigg\{\int_T^{\infty} \mu_F(T)  e^{- \mu_F (T) (t-T)}  \,G'\Big(\frac{C^o (T)}{-\dts\sup_{T\leq u'<t} \xi^{\star}_y (u')}\Big) dt \,\Big| {\cal F}_T\bigg\}.
\]
Uniqueness of $\xi^{\star}_y$ implies $\dts\sup_{T\leq u'<t} \xi^{\star}_y (u') = \mbox{const}$ for all $t>T$ and such constant is $-\frac{1}{y}$ by (\ref{xi-star after T}).
Now take $\tau \in \U[0,T]$  in (\ref{equation xi up to infty}), then 
\bey\label{equation application B-ElK}
& & \hspace{-0.8cm} e^{-\int_0^{\tau} \mu_F (r)\,dr} \frac{C^o(\tau)}{f_C(\tau)}\ind_{[0,T)} (\tau) 
= E\bigg\{\int_{\tau}^{T} e^{-\int_0^t \mu_F (r)\,dr}  C^o (t) \tR_C \Big( \frac{C^o (t)}{- \dts\sup_{\tau\leq u'<t} \xi^{\star}_y (u')}, ~w(t), r(t) \Big) \,dt \nonumber\\
& & \hspace{4cm}  
+ e^{- \int_0^T \mu_F (r)\,dr}  C^o (T) \int_{T}^{\infty} \mu_F(T)  e^{-\mu_F (T) (t-T)}
 G' \Big(\frac{C^o (T)}{- \dts\sup_{\tau\leq u' <t} \xi^{\star}_y (u')} \Big) dt \Big| {\cal F}_{\tau} \bigg\}\nonumber\\
& &\hspace{3.8cm}  =  E\bigg\{\int_{\tau}^{T} e^{-\int_0^t \mu_F (r)\,dr}  C^o (t) \tR_C \Big( \frac{C^o (t)}{- \dts\sup_{\tau\leq u'<t} \xi^{\star}_y (u')}, ~w(t), r(t) \Big) \,dt \\
& & \hspace{1.2cm}  
+ e^{- \int_0^T \mu_F (r)\,dr}  C^o (T) \int_{T}^{\infty} \mu_F(T)  e^{-\mu_F (T) (t-T)}
 G' \Big( \frac{C^o (T)}{\Big[({- \dts\sup_{\tau\leq u' < T}} \xi^{\star}_y (u'))  \wedge \frac{1}{y}\Big]\ind_{[0,T)}(\tau)
 } \Big) dt \Big| {\cal F}_{\tau} \bigg\}\nonumber
 \eey
and integrating the exponential in the last integral provides (\ref{existence of xi star}).
Notice that  in the last line $G'$ is zero for all $\omega$ such that $\tau(\omega) \notin [0,T)$, since $G' (+\infty) =0$, and the indicator is a reminder of the requirement $\tau(\omega) <T$ for the validity of the equation. 
\end{proof}
Now we set 
\beq\label{base capacity by xi star}
l^{\star}_y (t) := -\frac{C^o (t)}{\xi^{\star}_y(t)},\qquad t\in [0,T).
\eeq
The stricly positive process $\l^{\star}_y(t)$ is the {\em base capacity} of the present model (see \cite{RS} for the original Definition 3.1). The reason for being so called will be clear soon after the Theorem below.
\begin{corollary}
\label{equation l-star}
Under Assumption-{\rm [I]}, for $y>0$  the optional process $l^{*}_y (t)$ is the unique upper right-continuous, positive solution of the representation problem
\bey\label{equation l star}
& &   E\bigg\{\int_{\tau}^{T} e^{-\int_0^t \mu_F (r)\,dr}  C^o (t) \tR_C \Big( C^o (t)
\sup_{\tau \leq u' < t}  \frac{l^{*}_y (u')}{C^o (u')}, w(t), r(t) \Big)dt \\
& & \hspace{0.5cm} +e^{-\int_0^T \mu_F (r)\,dr} C^o (T) G' \Big( \frac{C^o (T)}{\ind_{[0,T)}(\tau)} \Big[\dts\sup_{\tau \leq u' <T} \frac{l^{*}_y (u')}{C^o (u')}\vee y\Big] \Big) ~\Big|~ {\cal F}_{\tau} \bigg\} 
= e^{-\int_0^{\tau} \mu_F (r)\,dr} \frac{C^o (\tau)}{f_C(\tau)} \ind_{[0,T)}(\tau), \nonumber
\eey
for all $\tau \in \U[0,T]$. 
\end{corollary}
\begin{proof}
The proof is straightforward after plugging (\ref{base capacity by xi star}) into (\ref{existence of xi star}) (see also \cite{ChF},  Lemma 4.1). In fact, $\frac{1}{- \sup_{\tau\leq u'<t} \xi^{\star}_y (u')} = \frac{1}{\inf_{\tau\leq u'<t} \frac{C^o(u')}{l^{\star}_y (u')}}= \sup_{\tau\leq u'<t} \frac{l^{\star}_y (u')}{C^o(u')}$ for any $t<T$.
\end{proof}
At this point, a more careful look at (\ref{equation l star}) brings back to mind the supergradient (\ref{supergradient}) and we can finally solve the very general original capacity expansion problem, for each initial capacity $y$. In fact, the optimal control investment process may be obtained in terms of the base capacity $\l^{\star}_y(t)$ (cf. the simpler case of \cite{ChF} with no scrap value). 
\begin{theorem}
\label{control by base capacity}
Under Assumption-{\rm [I]}, the stochastic control process
\beq\label{optimal nu_l}
{\nu}_{l^*}^{y}(t):=\int_{[0,t)}\frac{C^o(u)}{f_C(u)}\,d\overline \nu_{l^*}^{y}(u),  \qquad t\in[0,T),
\eeq
with
\beq
\label{overline nu_l}
\left\{\matrix{ 
 \overline\nu_{l^*}^{y}(t) := \dts\sup_{0\leq u'<t}\Big[ \frac{l^{*}_y (u')}{C^o (u')} \vee y \Big] -y = \sup_{0\leq u' < t} \Big[\frac{l^{*}_y (u')}{C^o(u')} - y  \Big]^+,  \qquad t\in (0,T),\hfill\cr  
 \overline\nu_{l^*}^{y}(0) :=0. \hfill\cr  
}
\right.
\eeq 
solves the first-order conditions {\rm (\ref{FOC}) }. Hence ${\nu}_{l^*}^{y}(t)$ is the unique optimal investment process of problem (\ref{cap prob from 0}). 
\end{theorem}
\begin{proof}
Recall (\ref{C solution}), then 
$C^{y}(t;\nu_{l^*}^{y}) = C^o(t)[y + {\overline\nu_{l^*}^{y}}(t)] = C^o(t)\sup_{0\leq u'<t}[\frac{l^{*}_y (u')}{C^o (u')} \vee y]$. Then, for all $\tau \in \U[0,T]$ and $\omega \in \Omega$ such that $\tau (\omega) < T$, 
\Bey
& &\hspace{-0.3cm} E\bigg\{ \int_{\tau}^T  e^{-\int_0^t \mu_F(r)\,dr}C^o(t) \tR_C(C^{y}(t;\nu_{l^*}^{y}), w(t), r(t))  \,dt +  e^{-\int_0^T \mu_F(r)\,dr} C^o(T) G'(C^{y}(T;\nu_{l^*}^{y}))  \,\bigg|\,\cF_{\tau}\bigg\} \\
& & \leq E\bigg\{ \int_{\tau}^T  e^{-\int_0^t \mu_F(r)\,dr}C^o(t) \tR_C\Big(C^o(t)\sup_{\tau\leq u'<t}\frac{l^{*}_y (u')}{C^o (u')}, w(t), r(t)\Big)  \,dt \\
& & \hspace{1cm} +  e^{-\int_0^T \mu_F(r)\,dr} C^o(T) G'\Big(C^o(T)\sup_{\tau \leq u' <T} \Big[\frac{l^{*}_y (u')}{C^o (u')} \vee y\Big]\Big) \,\bigg|\,\cF_{\tau}\bigg\} = e^{-\int_0^{\tau} \mu_F (r)\,dr} \frac{C^o (\tau)}{f_C(\tau)}
\Eey
by the decreasing property of $R_C$ and $G'$ in $C$ and by (\ref{equation l star}).
Hence $\nabla_{\nu} \cJ_{y} (\nu_{l^*}^{y})(\tau) \leq 0$. 

Moreover, at times $\tau$ of strict increase for $\overline\nu_{l^*}^{y}$, the above holds with equality since $C^{y}(t;\nu_{l^*}^{y}) = C^o(t)\sup_{\tau\leq u'<t}\frac{l^{*}_y (u')}{C^o (u')}$ for $\tau <t \leq T$. 
Therefore $\nabla_{\nu} \cJ_{y} (\nu_{l^*}^{y})(\tau) = 0$ when $d \overline\nu_{l^*}^{y}(\tau) >0$ and (\ref{FOC}) follows.
\end{proof}
For a firm starting at time $0$ from a capacity level $y$ lower than $l^{*}_y (0)$, it is easy to see from 
(\ref{overline nu_l}) that $\overline\nu_{l^*}^{y}(0+)$ is the jump size required to instantaneously reach the optimal capacity level $l^{*}_y (0)$. 
Similarly, the base capacity $\l^{\star}_y(t)$ represents the optimal capacity level at time $t$ for a firm starting at time $0$ from capacity $y$. For that reason the optimal capacity process $C^{y}(t;\nu_{l^*}^{y})$ can be referred to as {\em the capacity process that tracks} $l^{*}_y$.  

There is no closed form solution for the solution $\l^{\star}_y(t)$ of equation (\ref{equation l star}), but it might be found   numerically by backward induction on discretized version of the equation, as suggested by Bank and Follmer \cite{BF}, Section 4, for their model.

\section{The case of deterministic coefficients}
\label{investment exercise boundary} 
When the following\\
{\bf Assumption-[det]} 
\beq\label{AssCdet} 
\mu_C, \sigma_C, f_C, \mu_F, w, r \mathrm{~~are~deterministic~functions}
\eeq
holds, the singular stochastic control problem of capacity expansion may be studied by variational methods through the optimal stopping problem associated to it, the so-called {\em optimal risk of not investing until time $t$}, and the free boundary ${\hat y} (t)$ arising from it. 
It is then natural to ask whether and how the base capacity $l^{*}_y$ (which is stochastic and depends on the initial capacity $y$) is linked to the free boundary ${\hat y}$.  In  \cite{ChF}  it was shown that  $l^{*}(t) ={\hat y}(t)$ but there $l^{*}$ did not depend on $y$ as the model had no scrap value at the terminal time $T$.  In the present model with scrap value, some further conditions are needed in order to obtain the equality, and hence the independence  of the base capacity from $y$.

It can be shown that the capacity problem may be imbedded allowing the capacity process to start at any time $s \in [0,T]$ from level $y>0$ (by generalizing, among others, the model with scrap value, constant coefficients and additive production function depending on labor in \cite{CHi}, the model with constant coefficients but no scrap value in  \cite{CH1},  the model with deterministic time-dependent coefficients but no scrap value, no labor, no interest rate  in \cite{ChF}) . The details are provided in Appendix~\ref{variational approach to free boundary} for completeness, but here we recall the main points. First of all, the value function of the firm's optimal capacity problem $V(s,y)$ (see (\ref{cap prob from s}) as compared to (\ref{cap prob from 0})) and the value function $v(s,y)$ of the associated optimal stopping problem are linked; in fact, $v(t,y)$ is {\em the shadow value of installed capital}, that is $v(s,y) = \frac{\partial}{\partial y} V(s,y)$ (cf. (\ref{v derivative V at s})). 
The free boundary ${\hat y}(t)$ makes the strip $[0,T)\times (0, \infty)$ split into the {\em Continuation Region} 
\[
\Delta = \Big\{(s,y) \in [0,T) \times (0,\infty) : v(s,y) < \frac{1}{f_C (s)} \Big\}  
\]
where it is not optimal to invest as the capital's replacement cost is strictly greater than the shadow value of installed capital,  and its complement, the {\em Stopping Region} $\Delta^c$ where it is optimal to invest instantaneously (cf. (\ref{Delta}), (\ref{Delta complement})). 

Also, setting  $C^s (t):=\frac{C^o(t)}{C^o(s)} = e^{-\int_s^t \mu_C(r)dr} \cM_s (t)$ (cf. (\ref{martingale})) and $Y^{s,y}(t):=y\,C^s(t)$ for $s \in [0,T]$, $t\geq s$, $y>0$, and performing a change of probability measure allows to rewrite the ``optimal risk of not investing'' $v(s,y)$  under the new discount factor  $\bar\mu (t) := \mu_C (t) + \mu_F (t)\geq\ve_o>0 $ (cf.  (\ref{vv s})), and then prove that  $\hat y (s)$ is the lower boundary of $\Delta$ in the $(s,y)$-plane (cf. (\ref{hat y in function of v})), 
\[
\hat y (s) =\sup \Big\{z\geq 0:  v(s, z)= \frac{1}{f_C (s)}\Big\},
\]
and that the unique optimal investment process $\hat{\nu}^{s,y}$ of the capacity problem can be obtained in terms of $\overline\nu^{s,y}$  (cf. (\ref{bar nu s})), the left-continuous inverse of the optimal stopping time for $v(s,y)$. It follows that
\[
\overline\nu^{s,y}(t):=\dts\sup_{s\leq u<t}\Big[\frac{\hat y(u)}{C^s(u)} - y \Big]^+ \mbox{~~for~} t>s,
\]
with $\overline\nu^{s,y}(s) :=0$ and $C^{s,y}(t;\hat{\nu}^{s,y}) \geq \hat y(t)$ a.s.

Apparently the function ${\hat y}(t)$ might be related to the stochastic process $l^{\star}_y (t)$. In fact the controls 
$\overline\nu^{0,y}$  and $\overline\nu_{l^*}^{y}$ are similar and we know that the optimal investment process is unique, so by identifying the controls we would like to conclude that ${\hat y}(t)$ and $l^{\star}_y (t)$ are the same but, again, the latter depends on the initial capacity $y$! To handle this $y$-dependence some properties of $\hat y(t)$ are needed. 

We stress that all the effort in the rest of this Section is not needed in the case of no scrap value. We start by looking at the $s$-sections of $\Delta$ (cf. (\ref{Delta})).
\begin{proposition}
\label{v non-increasing in y}
Under Assumption-{\rm [det]}, for each $s\in [0,T)$ fixed,  \vspace{-0.3cm}\\

 ~~~{\rm (i)}   $~~~v(s, \cdot)$ is non-increasing in $y$; 
 
 ~~~{\rm (ii)}  $~~v(s, \cdot)$ is strictly decreasing on $y > \hat y (s)$;
 
 ~~~{\rm (iii)} ~the set $\{y>0\,:\,v(s,y)< 1/f_C(s)\}$  is connected.
\end{proposition}
\begin{proof}
Recall that $\tR(C,w,r)$ is strictly concave in $C$, hence $\tR_C$ is strictly decreasing in $C$, and $G'$ is non-increasing. Thus,  $Y^{s,y_1}(t)< Y^{s,y_2}(t)$ for $y_1<y_2$ implies that $v(s, y)$ is  non-increasing in $y$ and strictly decreasing for $y>\hat y(s)$.  
It follows that the set $\{y>0\,:\,v(s,y)< 1/f_C(s)\}$  is connected for each fixed $s$.
\end{proof}
On the other hand, the study of the $y$-sections of $\Delta$ is complicated  by the time dependence of $f_C$. The following result takes care of it.
Some of the arguments in the proof are similar to those needed to prove Theorem~\ref{stop} (see the discussion above it)
but this proof is much more tricky and requires the extra conditions below, so we provide it for completeness.
\begin{proposition}
\label{v non-increasing in s}
Under Assumption-{\rm [det]} and the further conditions
\beq\label{efficiency condition}
\left\{
\matrix{\tR_C(C,w,r) \mbox{ and } G'(C) \mbox{ are convex in~} C\in (0, \infty),\hfill\cr
\|\sigma_C\|^2 \leq \mu_C    \mbox{~ a.e.~in~} [0,T], \hfill\cr
f_C\in C^1([0,T])  \mbox{~ with ~} \bar\mu \leq -f'_C/f_C  \mbox{~ a.e.~on~} [0,T], \hfill\cr
}
\right.
\eeq
for  $y>0$ fixed,\vspace{-0.4cm}\\

 ~~{\rm(i)}~~  $~v(s, y) - \frac{1}{f_C(s)}$ is non-increasing in $s \in[0,T)$; 
 
 ~~{\rm(ii)}~~  the set  $\{s\geq 0: v(s,y)<1/f_C(s)\}$ is connected.
 \end{proposition}
\begin{proof}
If, for fixed $y>0$, we show  that $v(s, y) - \frac{1}{f_C(s)}$ is non-increasing in $s$, then point (ii) will follow from $v(s,y) - \frac{1}{f_C(s)} \leq 0$ (which is obtained by taking $\tau=s$ in (\ref{vv s})).
So we show point (i). Fix $s_1<s_2$, then  (see (\ref{vv s}))
\Bey
v(s_1,y) \hspace{-0.2cm}&=&\hspace{-0.2cm} \inf_{\tau\in\U_{s_1}[s_1,T]} E^{Q_{s_1}}\bigg\{\int_{s_1}^{\tau} e^{-\int_{s_1}^u \bar\mu (r)dr} \tR_C(Y^{s_1,y}(u), w(u), r(u))\,du \\
& & \hspace{2.6cm}+e^{-\int_{s_1}^{\tau} \bar\mu (r)dr} \frac{1}{f_C (\tau)}\ind_{\{\tau<T\}} 
+  e^{-\int_{s_1}^T \bar\mu (r) dr} G'(Y^{s_1,y}(T))\ind_{\{\tau=T\}} \bigg\}\\
 \hspace{-0.2cm}&=&\hspace{-0.2cm} \inf_{\tau\in\U_{s_1}[s_1,T]} E^{Q_{s_1}}\bigg\{\hspace{-0.1cm}
\ind_{\{\tau<s_2\}} \Big[\int_{s_1}^{\tau} e^{-\int_{s_1}^u \bar\mu (r) dr} \tR_C(Y^{s_1,y}(u), w(u), r(u)) du 
+e^{-\int_{s_1}^{\tau} \bar\mu (r) dr}\hspace{-0.1cm} \frac{1}{f_C(\tau)} \Big] \\
& & \hspace{2.6cm} 
+\ind_{\{s_2\leq\tau\leq T\}} \int_{s_1}^{s_2} e^{-\int_{s_1}^u \bar\mu (r)dr} \tR_C(Y^{s_1,y}(u), w(u), r(u))du \\
& & \hspace{2.6cm}  + \ind_{\{s_2\leq\tau\leq T\}} 
 \Big[ e^{-\int_{s_1}^{s_2}\bar\mu (r)\,dr} \Big(  \int_{s_2}^{\tau} e^{-\int_{s_2}^u \bar\mu (r)dr} \tR_C(Y^{s_1,y}(u), w(u), r(u))du \\
& & \hspace{2.6cm} + e^{-\int_{s_2}^{\tau} \bar\mu (r)dr} \frac{1}{f_C(\tau)}\ind_{\{\tau<T\}}  +  e^{-\int_{s_2}^T \bar\mu (r)dr} G'(Y^{s_1,y}(T))\ind_{\{\tau=T\}} \Big) \Big]\bigg\}.
\Eey
Set ${\bar \tau} := \tau \vee s_2$ and notice that
\Bey
& & \hspace{-1cm} \Big( \int_{s_2}^{\bar\tau} e^{-\int_{s_2}^u \bar\mu (r)dr} \tR_C(Y^{s_1,y}(u), w(u), r(u))du 
+e^{-\int_{s_2}^{\bar\tau} \bar\mu (r)dr} \frac{1}{f_C(\bar\tau)}\ind_{\{\bar\tau<T\}}  \\   
&  &\hspace{6.3cm}   +\,  e^{-\int_{s_2}^T \bar\mu (r)dr} G'(Y^{s_1,y}(T))\ind_{\{\bar\tau=T\}} \Big) \\
  & & = \ind_{\{s_2\leq\tau\leq T\}} \Big(\hspace{0.5cm} \Big){\Big|}_{{\bar \tau} =\tau }
+\ind_{\{s_1\leq\tau < s_2\}} \Big[ \frac{1}{f_C(s_2)}\ind_{\{s_2<T\}} + G'(Y^{s_1,y}(T))\ind_{\{s_2=T\}} \Big].
\Eey
Then
\Bey
& & \hspace{-1cm} v(s_1,y) = \inf_{\tau\in\U_{s_1}[s_1,T]} E^{Q_{s_1}}\bigg\{
\ind_{\{\tau<s_2\}} \Big[\int_{s_1}^{\tau} e^{-\int_{s_1}^u \bar\mu (r)\,dr} \tR_C(Y^{s_1,y}(u), w(u), r(u))\,du 
 +e^{-\int_{s_1}^{\tau} \bar\mu (r)\,dr} \frac{1}{f_C(\tau)} \Big] \\
& & \hspace{3.5cm} +\ind_{\{s_2\leq\tau\leq T\}} \int_{s_1}^{s_2} e^{-\int_{s_1}^u \bar\mu (r)\,dr} \tR_C(Y^{s_1,y}(u), w(u), r(u)) \,du \\
& & \hspace{3.5cm} \, + e^{-\int_{s_1}^{s_2} \bar\mu (r)\,dr} 
 \Big[\int_{s_2}^{\bar\tau} e^{-\int_{s_2}^u \bar\mu (r)\,dr} \tR_C(Y^{s_1,y}(u), w(u), r(u))\,du \\
& & \hspace{4.8cm} \,
 +e^{-\int_{s_2}^{\bar\tau} \bar\mu (r)\,dr} \frac{1}{f_C(\bar\tau)}\ind_{\{{\bar\tau}<T\}}  +  e^{-\int_{s_2}^T \bar\mu (r)\,dr} G'(Y^{s_1,y}(T))\ind_{\{\bar\tau=T\}} \\
& & \hspace{5cm} - \,\frac{1}{f_C(s_2)}\ind_{\{\tau < s_2<T\}} - \, G'(Y^{s_1,y}(T))\ind_{\{\tau < s_2=T\}} 
\Big]\bigg\} \\
& & \geq 
 \inf_{\tau\in\U_{s_1}[s_1,T]} E^{Q_{s_1}}\bigg\{
\ind_{\{\tau<s_2\}} \Big[\int_{s_1}^{\tau} e^{-\int_{s_1}^u \bar\mu (r)\,dr} \tR_C(Y^{s_1,y}(u), w(u), r(u))\,du \\
& & \hspace{1.4cm} +e^{-\int_{s_1}^{\tau} \bar\mu (r)\,dr} \frac{1}{f_C(\tau)}  - e^{-\int_{s_1}^{s_2}\bar\mu (r)\,dr} \frac{1}{f_C(s_2)}\ind_{\{s_2<T\}} - e^{-\int_{s_1}^{s_2}\bar\mu (r)\,dr} G'(Y^{s_1,y}(T))\ind_{\{s_2=T\}} \Big] \\
& & \hspace{3.1cm} 
+\ind_{\{s_2\leq\tau\leq T\}} \int_{s_1}^{s_2} e^{-\int_{s_1}^u \bar\mu (r)\,dr} \tR_C(Y^{s_1,y}(u), w(u), r(u))\,du \bigg\}\\
& &  \hspace{0.7cm} +  \inf_{\tau\in\U_{s_1}[s_1,T]} E^{Q_{s_1}}\bigg\{ e^{-\int_{s_1}^{s_2} \bar\mu (r)\,dr} 
 \Big[\int_{s_2}^{\bar\tau} e^{-\int_{s_2}^u \bar\mu (r)\,dr} \tR_C(Y^{s_1,y}(u), w(u), r(u))\,du   \\
& & \hspace{3.8cm} +e^{-\int_{s_2}^{\bar\tau} \bar\mu (r)\,dr} \frac{1}{f_C(\bar\tau)}\ind_{\{{\bar\tau}<T\}}
  +  e^{-\int_{s_2}^T \bar\mu (r)\,dr} G'(Y^{s_1,y}(T))\ind_{\{\bar\tau=T\}}  
\Big]\bigg\}.
\Eey
Assumption-[G] provides an upper bound on the second negative term and hence
\Bey
v(s_1,y) &\geq& 
\inf_{\tau\in\U_{s_1}[s_1,T]} E^{Q_{s_1}}\bigg\{\ind_{\{\tau<s_2\}} \int_{s_1}^{\tau} \hspace{-0.1cm} e^{-\int_{s_1}^u \bar\mu (r)\,dr} \tR_C(Y^{s_1,y}(u), w(u), r(u))\,du \\
& & \hspace{2.7cm} +\ind_{\{\tau<s_2\}} \Big[ e^{-\int_{s_1}^{\tau} \bar\mu (r)\,dr} \frac{1}{f_C(\tau)} 
 - e^{-\int_{s_1}^{s_2}\bar\mu (r)\,dr} \frac{1}{f_C(s_2)} \Big] \bigg\}\\
& &  +  e^{-\int_{s_1}^{s_2} \bar\mu (r)\,dr} \hspace{-0.2cm} \inf_{\tau\in\U_{s_1}[s_1,T]} E^{Q_{s_1}}\bigg\{ 
 \Big[\int_{s_2}^{\bar\tau} e^{-\int_{s_2}^u \bar\mu (r)\,dr} \tR_C(Y^{s_1,y}(u), w(u), r(u))\,du   \\
& & \hspace{2.5cm} +e^{-\int_{s_2}^{\bar\tau} \bar\mu (r)\,dr} \frac{1}{f_C(\bar\tau)}\ind_{\{{\bar\tau}<T\}}
  +  e^{-\int_{s_2}^T \bar\mu (r)\,dr} G'(Y^{s_1,y}(T))\ind_{\{\bar\tau=T\}}  
\Big]\bigg\}.
\Eey
For the latter expectation we proceed as in the proof of Theorem \ref{stop}. We consider the canonical probability space  $(\bar\Omega,\bar P)$ where now $\bar \Omega = \cC_0 [s_1,T]$ (the space of continuous functions on $[s_1,T]$ which are zero at $s_1$) and $\bar P$ is the Wiener measure on $\bar \Omega$. 
We define $W^{Q_{s_1}}(t,\bar\omega)=\bar\omega (t)$ the coordinate mapping on $\cC_0 [s_1,T]$ with 
$\bar\omega=(\bar\omega_1,\bar\omega_2)$, where 
$\bar\omega_1=\{W^{Q_{s_1}}(u) - W^{Q_{s_1}}(s_1): s_1\leq u\leq s_2\},\ \bar\omega_2=\{W^{Q_{s_1}}(u)-W^{Q_{s_1}}(s_2): s_2\leq u\leq T\}$. 
Hence  $\bar P$ is a product measure on $\cC_0[s_1,T]=\cC_0[s_1,s_2]\times\cC_0[s_2,T]$, due to independence of the increments of $W^{Q_{s_1}}$. 
Then,  for each $\bar\omega_1 \in \bar\Omega$ fixed, ${\bar\tau}_{\bar\omega_1} (\newdot) := \tau ({\bar\omega_1}, \newdot) \vee s_2 \in \U_{s_2}[s_2,T]$ with $\tau(\bar\omega_1,\newdot)$  measurable w.r.t. $\cF_{s_2,T}$.
If $E^{{\bar P}_{s_i}}_{\bar\omega_i} \{\cdot\}$ denotes expectation over $\bar\omega_i$ and, for each ${\bar\omega_1} \in {\bar \Omega}$, the expectation over ${\bar\omega_2}$ is denoted by $\Phi (t, Y^{s_1,y}(t); {\bar\tau}_{\bar\omega_1})$ (as in the proof of Theorem \ref{stop}), then the last expectation above is written as $E^{{\bar P}_{s_1}}_{\bar\omega_1}\{\Phi (t, Y^{s_1,y}(t); {\bar\tau}_{\bar\omega_1})\}$.
Therefore for the last infimum above  it holds 
$\dts\inf_{\tau\in\U_{s_1}[s_1,T]} E^{{\bar P}_{s_1}}_{\bar\omega_1}\{\Phi (t, Y^{s_1,y}(t); {\bar\tau}_{\bar\omega_1})\}
\geq \dts\inf_{\tau' \in\U_{s_2}[s_2,T]} E^{{\bar P}_{s_1}}_{\bar\omega_1}\{\Phi (t, Y^{s_1,y}(t); \tau')\}$, 
and hence
\bey\label{geq 2 inf}
& & v(s_1,y) \geq 
\inf_{\tau\in\U_{s_1}[s_1,T]} E^{Q_{s_1}}\bigg\{\ind_{\{\tau<s_2\}} \int_{s_1}^{\tau} \hspace{-0.1cm} e^{-\int_{s_1}^u \bar\mu (r)\,dr} \tR_C(Y^{s_1,y}(u), w(u), r(u))\,du \nonumber\\
& & \hspace{4.6cm} +\ind_{\{\tau<s_2\}} \Big[ e^{-\int_{s_1}^{\tau} \bar\mu (r)\,dr} \frac{1}{f_C(\tau)} 
 - e^{-\int_{s_1}^{s_2}\bar\mu (r)\,dr} \frac{1}{f_C(s_2)} \Big] \bigg\} \\
& &  \hspace{2cm} +\, e^{-\int_{s_1}^{s_2} \bar\mu (r)\,dr}  \hspace{-0.2cm} \inf_{\tau'\in\U_{s_2}[s_2,T]} E^{Q_{s_1}}\bigg\{ 
\int_{s_2}^{\tau'} e^{-\int_{s_2}^u \bar\mu (r)\,dr} \tR_C(Y^{s_1,y}(u), w(u), r(u))\,du \nonumber \\
& & \hspace{4.5cm} +e^{-\int_{s_2}^{\tau'} \bar\mu (r)\,dr} \frac{1}{f_C(\tau')}\ind_{\{{\tau'}<T\}}
  +  e^{-\int_{s_2}^T \bar\mu (r)\,dr} G'(Y^{s_1,y}(T))\ind_{\{\tau'=T\}}  
\bigg\}. \nonumber
\eey

Now consider the second infimum in (\ref{geq 2 inf}). For the first term we have  
\Bey
& & E^{Q_{s_1}}\bigg\{
 \int_{s_2}^{\tau'} e^{-\int_{s_2}^u \bar\mu (r)\,dr} \tR_C(Y^{s_1,y}(u), w(u), r(u))\,du \bigg\} \\
 & & \hspace{1cm} = E^{Q_{s_1}}\bigg\{ 
\int_{s_2}^{\tau'} e^{-\int_{s_2}^u \bar\mu (r)\,dr} E^{Q_{s_1}} \Big\{\tR_C(Y^{s_1,y}(u), w(u), r(u)) \Big|\cF_{s_2, T}\Big\}\,du \bigg\}\\
& & \hspace{1cm} \geq E^{Q_{s_1}}\bigg\{
\int_{s_2}^{\tau'} e^{-\int_{s_2}^u \bar\mu (r)\,dr} \tR_C(E^{Q_{s_1}} \{Y^{s_1,y}(u) |\cF_{s_2, T}\}, w(u), r(u))\,du \bigg\}
\Eey
since $\tR_C$ is convex in $C$ by (\ref{efficiency condition}) and $Q_{s_1}|\cF_{s_2, T}$ is a regular conditional probability distribution. Moreover $C^{s_2} (\newdot)$ is $\cF_{s_2, T}$-measurable and $Y^{s_1,y}(s_2)$ is independent of $\cF_{s_2, T}$. Then, for $u>s_2$, using $\mu_C \geq \|\sigma_C\|^2$  a.e. (cf. (\ref{efficiency condition})$_2$) we get
\Bey 
 E^{Q_{s_1}} \{Y^{s_1,y}(u) |\cF_{s_2, T}\} 
\hspace{-0.2cm} &=& \hspace{-0.2cm} E^{Q_{s_1}} \{Y^{s_1,y}(s_2) \, C^{s_2} (u)|\cF_{s_2, T}\} 
= E^{Q_{s_1}} \{Y^{s_1,y}(s_2) |\cF_{s_2, T}\} \, C^{s_2} (u) \\
\hspace{-0.2cm} &=& \hspace{-0.2cm} E^{Q_{s_1}} \{Y^{s_1,y}(s_2)\} \, C^{s_2} (u) \\
\hspace{-0.2cm} &=& \hspace{-0.2cm} y\, C^{s_2}(u) 
E^{Q_{s_1}}\Big\{e^{- \int_{s_1}^{s_2} [\mu_C (u) - \|\sigma_C(u)\|^2  ]\,du} \cM^{Q_{s_1}}_{s_1} (s_2) 
\Big\} \leq y\, C^{s_2}(u) 
\Eey
where $\cM^{Q_{s_1}}_{s_1} (t) = e^{[\int_{s_1}^t \sigma_C^{\top}(r) dW^{Q_{s_1}}(r) - \frac{1}{2} \int_{s_1}^t \| \sigma_C(r)\|^2 dr]}$ is a martingale under $Q_{s_1}$ (cf. (\ref{martingale})).
Hence $\tR_C$ decreasing in $C$ implies
$\tR_C(E^{Q_{s_1}} \{Y^{s_1,y}(u) |\cF_{s_2, T}\}, w(u), r(u))
\geq \tR_C(y C^{s_2}(u), w(u), r(u))$,
and $Q_{s_1} = Q_{s_1, s_2} \otimes Q_{s_2}$ gives 
\Bey
& & E^{Q_{s_1}}\bigg\{
\int_{s_2}^{\tau'} e^{-\int_{s_2}^u \bar\mu (r)\,dr} \tR_C(Y^{s_1,y}(u), w(u), r(u))\,du \bigg\}\\
& & \hspace{1cm} \geq ~ E^{Q_{s_2}}\bigg\{ 
\int_{s_2}^{\tau'} e^{-\int_{s_2}^u \bar\mu (r)\,dr} \tR_C(y\, C^{s_2}(u), w(u), r(u))\,du \bigg\}
\Eey
being the last integral independent of $\cF_{s_1, s_2}$.

Similar arguments apply to the term involving $G'$ in (\ref{geq 2 inf}), using the fact that $G'$ is non-increasing and convex  by  (\ref{efficiency condition}). Therefore the last infimum in (\ref{geq 2 inf}) is greater or equal $v(s_2, y)$, and we have
\bey \label{new geq 2 inf}
& & \hspace{-1cm} v(s_1,y) \geq
  \inf_{\tau\in\U_{s_1}[s_1,T]} E^{Q_{s_1}}\bigg\{\ind_{\{\tau<s_2\}} \int_{s_1}^{\tau} \hspace{-0.1cm} 
 e^{-\int_{s_1}^u \bar\mu (r)\,dr} \tR_C(Y^{s_1,y}(u), w(u), r(u))\,du  \\
& & \hspace{2.5cm} +\ind_{\{\tau<s_2\}} \Big[ e^{-\int_{s_1}^{\tau} \bar\mu (r)\,dr} \frac{1}{f_C(\tau)} 
 - e^{-\int_{s_1}^{s_2}\bar\mu (r)\,dr} \frac{1}{f_C(s_2)} \Big] \bigg\} 
 + e^{-\int_{s_1}^{s_2} \bar\mu (r)\,dr}\, v(s_2,y). \nonumber
\eey
Now subtracting $\frac{1}{f_C(s_1)}$ from both sides, adding and subtracting $e^{-\int_{s_1}^{s_2} \bar\mu (r)\,dr} \frac{1}{f_C(s_2)} $ on the right-hand side, and recalling that $\bar\mu>0$ give
\Bey
\hspace{-0.8cm} v(s_1,y) -\frac{1}{f_C(s_1)} &\geq& \hspace{-0.3cm}  \inf_{\tau\in\U_{s_1}[s_1,T]} E^{Q_{s_1}}\bigg\{ \ind_{\{\tau<s_2\}} \int_{s_1}^{\tau}   e^{-\int_{s_1}^u \bar\mu (r)\,dr} \tR_C(Y^{s_1,y}(u), w(u), r(u))\,du \\
& &  \hspace{2.5cm} +\ind_{\{\tau<s_2\}} \Big[ e^{-\int_{s_1}^{\tau} \bar\mu (r)\,dr} \frac{1}{f_C(\tau)} 
 - e^{-\int_{s_1}^{s_2}\bar\mu (r)\,dr} \frac{1}{f_C(s_2)} \Big] \bigg\} \\
& &  -\frac{1}{f_C(s_1)}  +e^{-\int_{s_1}^{s_2} \bar\mu (r)\,dr} \frac{1}{f_C(s_2)} +e^{-\int_{s_1}^{s_2} \bar\mu (r)\,dr} \Big[v(s_2,y) - \frac{1}{f_C(s_2)} \Big] \\
&=& \inf_{\tau\in\U_{s_1}[s_1,T]}   E^{Q_{s_1}}\bigg\{\ind_{\{\tau<s_2\}} \int_{s_1}^{\tau} e^{-\int_{s_1}^u \bar\mu (r)\,dr}  \tR_C(Y^{s_1,y}(u), w(u), r(u))\,du \\
&  &    + e^{-\int_{s_1}^{\tau\wedge s_2} \bar\mu (r)\,dr}\hspace{-0.1cm}  \frac{1}{f_C(\tau\wedge s_2)}  -\frac{1}{f_C(s_1)} \bigg\}  
  +e^{-\int_{s_1}^{s_2} \bar\mu (r)\,dr} \Big[v(s_2,y) - \frac{1}{f_C(s_2)} \Big].
\Eey
By condition (\ref{efficiency condition})$_3$  the last expectation above is non-negative, hence  $e^{-\int_0^{s} \bar\mu (r)\,dr} \Big[v(s,y) - \frac{1}{f_C(s)} \Big]$ 
is non-increasing in $s$, but the exponential is decreasing and $v(s,y) - \frac{1}{f_C(s)} \leq 0$ (by (\ref{vv s})), therefore $v(s,y) - \frac{1}{f_C(s)}$ itself must be non-increasing in $s$ and point (i) is proved.
\end{proof}
Now the above results enable us to prove the following Theorem that will be needed in Section~\ref{unifying views}, Theorem~\ref{base capacity indepent of y} to  identify the free boundary with the base capacity in the present time-inhomogeneous model with scrap value at the terminal time $T$. It is a generalization of \cite{CHi}, Proposition 3.3 which was limited to the constant coefficients case.  

\begin{theorem}
\label{boundary decreasing}
Under Assumption-{\rm[det]} and conditions {\rm (\ref{efficiency condition})}, 

~~{\rm(i)}~~~~  $\hat y(s)$ is non-increasing on $[0,T)$;

~~{\rm(ii)}~~~  $\hat y(s)$ is strictly positive on $[0,T)$  if also 
Assumption-{\rm [I]}$_1$ of Section~\ref{Bank-El Karoui approach} holds.
\end{theorem}
\begin{proof}
Point (i) follows from (\ref{hat y in function of v}) and point (i) of Proposition \ref{v non-increasing in s}. In fact, $0\geq v(s_1,y) - \frac{1}{f_C(s_1)} \geq v(s_2,y) - \frac{1}{f_C(s_2)}$ for $s_2>s_1$ implies $\hat y(s_2)\leq \hat y(s_1)$.
The lower boundary of the Borel set $\Delta$, the continuation region of problem (\ref{vv s}), is graph$(\hat y)$, 
and $\Delta$  lies above it. 
The first exit time of $(t,Y^{s,y}(t))$ from $\Delta$ is ${\hat\tau}(s,y)$ by (\ref{tau*bdy}).
To prove point (ii), suppose to the contrary that Assumption-{\rm [I]}$_1$ holds and $\hat y(t)=0$ for some $t<T$, then $\hat y(s) \equiv 0$ for $s \in [t,T)$ since $\hat y$ is non-increasing. Then by the dynamics of $Y^{s,y}$ it follows that  ${\hat\tau}(s,y) =T$ for all $(s,y) \in [t,T)\times (0, \infty)$ and hence
\[
 v(s,y)=E^{Q_{s}}\bigg\{\int_{s}^T e^{-\int_{s}^u \bar\mu (r)\,dr} \tR_C(Y^{s,y}(u), w(u), r(u))\,du
+  e^{-\int_{s}^T \bar\mu (r)\,dr} G'(Y^{s,y}(T)) \bigg\} <\frac{1}{f_C (s)},
\]
which is impossible since the expected value blows up as $y\downarrow 0$ by Assumption-[I]$_1$. 
\end{proof}
As a byproduct we obtain continuity of the optimal investment process.
\begin{corollary}
\label{nu continuous}
Under Assumption-{\rm[det]} and conditions {\rm (\ref{efficiency condition})}, 
the optimal investment process $\hat\nu^{s,y}(t)$ is continuous except possibly for an initial jump, 
 hence so is the optimal capacity process  $C^{s,y}(t;\hat\nu^{s,y})$.
\end{corollary}
\begin{proof}
 Recall that ${\hat\tau}(s,y)$ is non-decreasing in $y$ a.s. as pointed out below (\ref{bar nu s}). 
Let $z>y>\hat y(s)$, then the connectness properties proved in Proposition \ref{v non-increasing in y} and Proposition \ref{v non-increasing in s} imply $(s,z), (s,y) \in \Delta$ and ${\hat\tau}(s,y)>s$. Also $(t,Y^{s,y}(t))$ lies strictly below $(t,Y^{s,z}(t))$ since $Y^{s,z}(t)-Y^{s,y}(t)=(z-y)C^s(t)>0$. Hence at $u={\hat\tau}(s,y)$ the process $(u,Y^{s,y}(u))$ lies in the boundary, but $(u,Y^{s,z}(u))$ still lies in the interior of $\Delta$ since its boundary is non-increasing. It follows that ${\hat\tau}(s,y)<{\hat\tau}(s,z)$. So ${\hat\tau}(s,y)$ is strictly increasing a.s. on $y>\hat y(s)$, therefore
its left-continuous inverse (modulo a shift) $\overline\nu^{s,y}$  is continuous except possibly for an initial jump, and so is $C^{s,y}(t;\hat\nu^{s,y})$.
\end{proof}

\section{A unifying view on the optimal investment boundary}
\label{unifying views}
Under Assumption-[I] of Section~\ref{Bank-El Karoui approach}, Assumption-[det] of Section~\ref{investment exercise boundary}  and conditions (\ref{efficiency condition}) of Proposition \ref{v non-increasing in s}, in this Section we manage to show that the investment exercise boundary obtained by variational methods coincides with the base capacity provided by the Representation Theorem approach. Hence in the presence of scrap value at the terminal time $T$ getting a unifying view on the curve at which it is optimal to invest is possible but it requires adding extra conditions.   
 
Certainly  uniqueness of the optimal control, Proposition \ref{hnu s.thm} and Theorem \ref{control by base capacity} 
imply the identification of $\hat{\nu}^{0,y}(t)$ with the optimal control ${\nu}_{l^{\star}}^{y}(t)$ obtained via the Bank and El Karoui Representation Theorem.  
Moreover $\hat y(t)$ is strictly positive and its non-increasing property implies upper right-continuity \`{a} la Bourbaki, properties naturally enjoyed by the base capacity $l^{\star}_{y} (t)$.

Recalling (\ref{vv s}), writing (\ref{Gamma y,xi}) under the new probability measure $Q_0$ and taking care of the conditioning by arguments as in the proof  of Theorem \ref{stop} (see also \cite{ChF}, Proposition 5.2) provide \beq\label{Gamma -1/y same as vv s}
\hspace{0.2cm} \ti\Gamma^{y;-\frac{1}{y}} (t):= e^{\int_0^t \mu_F (r) dr} \dts\frac{1}{C^o(t)} \Gamma^{y;-\frac{1}{y}} (t) = v(t,yC^o(t)) \leq \frac{1}{f_C(t)}, \quad t\in (0,T), ~\mbox{ a.s.} 
\eeq
but, contrary to the no-scrap value case of \cite{ChF}, Proposition 5.3, $\ti\Gamma^{y;\xi} (t)$ cannot be written in terms of $v$ for a generic value of $\xi$ due to the $y$-parameter dependence of $\Gamma^{y;\xi}$ (carried over to $\xi^{\star}_{y}(t)$ and $l^{\star}_{y} (t)$) which accounts for the term $G'(y C^o(T))$ due to the scrap value.  Therefore (cf. (\ref{base capacity by xi star})),  
for $t\in [0,T)$,
\beq\label{l star compared to v}
\hspace{0.5cm}
l^{\star}_{y} (t) := -\frac{C^o(t)}{\xi^{\star}_{y} (t)} = \frac{C^o(t)}{-\sup\Big\{\xi<0 : \ti\Gamma^{y;\xi} (t) = \frac{1}{f_C(t)} \Big\}} = \sup\Big\{\hspace{-0.1cm} -\frac{1}{\xi} C^o(t) >0 : \ti\Gamma^{y;\xi} (t) = \frac{1}{f_C(t)} \Big\}
\eeq
  does not appear immediately linked to $\hat y(t)$  (cf. (\ref{hat y in function of v})).
So some work is needed in order to compare the two approaches. Notice that, for $t\in [0,T)$ and $z>y>0$, $l^{\star}_{y}(t) \geq l^{\star}_{z}(t)$ since $\ti\Gamma^{y;\xi} (t) \geq  \ti\Gamma^{z;\xi} (t)$ by the decreasing property of  $G'(\cdot)$, which together with the decreasing property of $\tR_C(\cdot,w, r)$ also implies 
\[
\frac{1}{f_C(t)} \geq  \ti\Gamma^{y;-\frac{1}{y}} (t) \geq \ti\Gamma^{y;-\frac{1}{z}} (t) \geq \ti\Gamma^{z;-\frac{1}{z}} (t). 
\]
The main result of this Section is the following
\begin{theorem}\label{base capacity indepent of y}
Under Assumption-{\rm [I]},  Assumption-{\rm [det]} and conditions (\ref{efficiency condition}),  for $t\in [0,T)$, the process  $l^{\star}_{y}(t) $ equals $\hat y(t)$ a.s., hence it may be assumed deterministic and independent of $y$.
\end{theorem}
\begin{proof}
If $\omega$ is such that $0<yC^o(\omega, t) \leq {\hat y} (t)$ (cf. (\ref{hat y in function of v})), 
then $\ti\Gamma^{y;-\frac{1}{y}}(\omega,t) = \frac{1}{f_C(t)}$  and 
$\ti\Gamma^{y;-\frac{1}{z}} (\omega,t) = \frac{1}{f_C(t)}$ for all $0<z<y$ as well. 
Therefore
\Bey
& & l^{\star}_{y}(\omega,t) = \frac{C^o(\omega,t)}{-\sup\Big\{-\frac{1}{z}\in [-\frac{1}{y}, 0) : \ti\Gamma^{y;-\frac{1}{z}} (\omega,t) = \frac{1}{f_C(t)} \Big\}} = \frac{C^o(\omega,t)}{\inf\Big\{\frac{1}{z}\in (0, \frac{1}{y}] : \ti\Gamma^{y;-\frac{1}{z}} (\omega,t) = \frac{1}{f_C(t)} \Big\}} \\
& &  = C^o(\omega,t) \sup\Big\{z \geq y : \ti\Gamma^{y;-\frac{1}{z}} (\omega,t) = \frac{1}{f_C(t)} \Big\}
\, \geq \, C^o(\omega,t)\sup\Big\{z \geq y : v(t, zC^o(\omega,t))= \frac{1}{f_C(t)} \Big\}. 
\Eey
Also
\Bey
& & \Big\{zC^o(\omega,t) \geq yC^o(\omega,t) : v(t, zC^o(\omega,t))= \frac{1}{f_C(t)} \Big\} 
\subset \Big\{z' \geq yC^o(\omega,t) : v(t, z')= \frac{1}{f_C(t)} \Big\}\\
& & \hspace{0.5cm} = \Big\{ \ti z' (\omega) C^o(\omega,t)  \geq yC^o(\omega,t) : 
v(t, \ti z' (\omega) C^o(\omega,t))= \frac{1}{f_C(t)}\Big\} 
\Eey 
with $\ti z' (\omega) := \frac{z'}{C^o(\omega,t)}$. Clearly the last set is contained into the first one, hence all sets coincide and
\beq\label{first step l-star}
 l^{\star}_{y}(\omega,t) \geq \sup\Big\{z' \geq yC^o(\omega,t) : v(t, z')= \frac{1}{f_C(t)} \Big\} = \hat y (t).
\eeq
We claim $l^{\star}_{y}(\omega,t) = \hat y (t)$. In fact assume not, then there exists $z_o > {\hat y}(t)$ such that $\ti\Gamma^{y;-\frac{C^o(\omega,t)}{z_o}} (t) = \frac{1}{f_C(t)}$. 
It follows that $\ti\Gamma^{y;-\frac{1}{z}} (\omega,t) = \frac{1}{f_C(t)}$
for all $\frac{\hat y (t)}{C^o(\omega,t)} <z \leq \frac{z_o}{C^o(\omega,t)}$ since $\tR_C$ and $G'$ are decreasing in $C$. As $\ti\Gamma^{y;-\frac{1}{z}} (\omega,t)$ is obtained at $\tau^{y;-\frac{1}{z}}(\omega,t)=t$, it coincides with the infimum over  $\tau \in [t,T)$ (which does not involve $G'$, thus not $y$); that is,
\[
\frac{1}{f_C(t)} = \mbox{ess}\dts\inf_{\hspace{-0.5cm} t\leq\tau < T} \ti\gamma^{y;-\frac{C^o(\omega,t)}{z_o}} (\tau)
= \mbox{ess}\dts\inf_{\hspace{-0.5cm} t\leq\tau < T} \ti\gamma^{y;-\frac{1}{z}} (\tau)
= \mbox{ess}\dts\inf_{\hspace{-0.5cm} t\leq\tau < T} \ti\gamma^{z;-\frac{1}{z}} (\tau)
\]
where, for convenience, we have denoted $\ti\gamma^{y;\xi} (\tau)$
the argument of the $\mbox{ess}\dts\inf$ of the corresponding $\ti\Gamma$. 
Also, $\hat y (t) < z C^o(\omega,t)$ implies 
$\ti\Gamma^{z;-\frac{1}{z}} (\omega,t) =v(t, zC^o(\omega,t)) < \frac{1}{f_C(t)}$, and hence $\ti\Gamma^{z;-\frac{1}{z}} (\omega,t)$ is not obtained before $T$  for all $zC^o(\omega,t) \in (\hat y (t), z_o]$ (cf. (\ref{tau*bdy})). In other words,  the process $z C^o(\omega,t) C^t(u)$, for $u>t$, never reaches the boundary $\hat y(u)$ before $T$, and that may only happen if $\hat y (u)\equiv 0$ for $u>t$, contraddicting point (ii) of Theorem \ref{boundary decreasing},
or else  if $\omega$ is in a null set.

On the other hand, if $\omega$ is such that $yC^o(\omega, t) > {\hat y} (t)$, 
then  $\ti\Gamma^{y;-\frac{1}{z}} (\omega,t) \leq \ti\Gamma^{y;-\frac{1}{y}}(\omega,t)< \frac{1}{f_C(t)}$  for all $z>y$; 
as well as 
$\ti\Gamma^{y;-\frac{1}{y}} (\omega,t) \leq \ti\Gamma^{y;-\frac{1}{z}} (\omega,t) \leq \ti\Gamma^{z;-\frac{1}{z}} (\omega,t)< \frac{1}{f_C(t)} \, (= \ti\Gamma^{\frac{{\hat y}(t)}{C^o(\omega,t)} ;-\frac{C^o(\omega,t)}{{\hat y}(t)}} (t))$  for all $\frac{{\hat y}(t)}{C^o(\omega,t)} <z <y$. 
Therefore to determine $l^{\star}_{y} (\omega,t)$  it suffices to consider  $z \leq \frac{{\hat y}(t)}{C^o(\omega,t)}$; in fact,
\Bey
&  & l^{\star}_{y}(\omega,t)  = \frac{C^o(\omega,t)}{-\sup\Big\{-\frac{1}{z}\in (- \infty, -\frac{C^o(\omega,t)}{{\hat y}(t)}] : \ti\Gamma^{y;-\frac{1}{z}} (\omega,t) = \frac{1}{f_C(t)} \Big\}} \\
&  &= C^o(\omega,t) \sup\Big\{z \leq \frac{{\hat y}(t)}{C^o(\omega,t)} : \ti\Gamma^{y;-\frac{1}{z}} (\omega,t) = \frac{1}{f_C(t)} \Big\}
\leq \sup\Big\{zC^o(\omega,t) \leq {\hat y}(t) : \ti\Gamma^{z;-\frac{1}{z}} (\omega,t) = \frac{1}{f_C(t)} \Big\}
\\
&  &
= \sup\Big\{zC^o(\omega,t) \leq {\hat y}(t) : v(t, zC^o(\omega,t)) =  \frac{1}{f_C(t)} \Big\},
\Eey
hence also
\beq\label{second step l-star}
l^{\star}_{y}(\omega,t) \leq \sup\Big\{z' \leq {\hat y}(t) : v(t, z') = \frac{1}{f_C(t)} \Big\} \, =\,  \hat y (t).  
\eeq
To prove $l^{\star}_{y}(\omega,t) = \hat y (t)$ assume not, then for $z_o \in (l^{\star}_{y}(\omega,t),\hat y (t)]$ 
it holds $\ti\Gamma^{y;-\frac{C^o(\omega,t)}{z_o}} (\omega,t) < \frac{1}{f_C(t)}$ 
and $\ti\Gamma^{\frac{z_o}{C^o(\omega,t)};-\frac{C^o(\omega,t)}{z_o}} (\omega,t)=v(t, z_o) = \frac{1}{f_C(t)}$ 
with  $\tau^{\frac{z_o}{C^o(\omega,t)};-\frac{C^o(\omega,t)}{z_o}}(\omega,t)=t$. It follows that must necessarily be $\tau^{y;-\frac{C^o(\omega,t)}{z_o}}(\omega,t)=\inf\Big\{u\in [t, T)\! : \ti\Gamma^{y;-\frac{C^o(\omega,t)}{z_o}} (\omega,u) = \frac{1}{f_C(u)} \Big\}\wedge T =T$ (see (\ref{equation tau xi}) and (\ref{best time up to T})), since the two $\ti\Gamma$ differ only for the scrap value, 
and this is a contraddiction since
\[ 
\mbox{ess}\dts\inf_{\hspace{-0.5cm} t\leq\tau < T} \ti\gamma^{y;-\frac{C^o(\omega,t)}{z_o}} (\tau)
\, = \,\mbox{ess}\dts\inf_{\hspace{-0.5cm} t\leq\tau \leq T} \ti\gamma^{\frac{z_o}{C^o(\omega,t)};-\frac{C^o(\omega,t)}{z_o}} (\tau)
\, = \,\frac{1}{f_C (t)}.
\]
\end{proof}
The above Theorem has an interesting consequence. The integral equation in Corollary \ref{equation l-star},  whose unique upper right-continuous, positive solution is $l^* (t)$, may be used to characterize uniquely the investment exercise boundary ${\hat y}(t)$ (see also \cite{ChF}, Theorem 5.5 for the no scrap value model). 
In fact, by writing the equation for $\tau=t \in [0,T)$ and 
$y = \frac{\hat y(t)}{C^o(t)}$, switching to the new probability measure $Q :=Q_{t}$, using the fact that $\frac{C^o (u)}{C^o (t)}$ for $u>t$ is independent of ${\cal F}_t$, we obtain the following 
\begin{corollary}\label{integral equation hat y}
Under Assumption-{\rm [I]},  Assumption-{\rm [det]} and conditions (\ref{efficiency condition}), the investment exercise boundary $\hat y (t)$ is the unique upper right-continuous, strictly positive solution of the integral equation 
\bey\label{integral equation}
& & \hspace{-1cm} E^{Q_t}\bigg\{\int_{t}^{T} e^{-\int_{t}^u \bar\mu (r)\,dr} \tR_C \Big(
\sup_{t\leq u' < u} \hat y (u') C^{u'}(u), ~w(u), r(u) \Big) \,du \\
& & \hspace{2.9cm} + e^{-\int_{t}^T \bar\mu (r)\,dr} G' \Big(\sup_{t\leq u' < T}  \hat y (u') C^{u'}(T)\Big) \bigg\} ~=~ \frac{1}{f_C(t)},\qquad \forall t\in[0,T). \nonumber
\eey
\end{corollary}
This is a useful result as it allows to find numerically the investment exercise boundary for any choice of ``cost'' functions $w$ and $r$, in a quite complex singular stochastic control problem of capacity expansion. 

In conclusion, putting together the steps, the following algorithm for the continuous optimal investment process $\hat{\nu}^{s,y}$ holds, \vspace{-0.2 cm}
\begin{enumerate}
\item given the cost functions $w(t), r(t)$, use equality (\ref{formula Appendix of[8]}) to calculate the partial $C$-derivative $\tR_C(C, w, r)$ of the reduced production function; \vspace{-0.7cm}\\
\item plug it into (\ref{integral equation}) and solve such integral equation (numerically after a discretization) to find its unique upper right-continuous, positive solution ${\hat y}(t)$;\vspace{-0.7cm}\\
\item plug ${\hat y}(t)$ into (\ref{bar nu by hat y}) to determine the continuous process $\overline\nu^{s,y}$; \vspace{-0.7cm}\\
\item finally plug $\overline\nu^{s,y}$ into (\ref{hnu s}) and obtain the optimal investment process $\hat{\nu}^{s,y}$.
\end{enumerate}

\appendix
\section{Variational approach review}
\label{variational approach to free boundary}
\renewcommand{\theequation}{A.\arabic{equation}}

Under Assumption-{\rm [det]} of Section~\ref{investment exercise boundary}, as in \cite{CH1}, \cite{CHi}, \cite{ChF},   the problem may be imbedded by considering the dynamics of the capacity process $C^{s,y}(t;\nu)$ starting at time $s \in [0,T]$ from $y>0$ and controlled by $\nu$, 
\beq
\label{C.eq starting at s}
\left\{
\matrix{
dC^{s,y}(t;\nu) = C^{s,y}(t;\nu) [-\mu_C(t)  dt + \sigma_C^{\top}(t) dW(t)] + f_C (t) d\nu(t),              
\qquad t\in (s, T], \hfill\cr
C^{s,y}(s;\nu) =y> 0, \hfill\cr
}
\right.
\eeq
with $\nu \in {\cal S}_s \! := \! \{ \nu: [s,T] \rightarrow \real: \!\!\! \mbox{\rm ~non-decreasing, lcrl, adapted  process s.t.} \,\nu(s)=0 \!\mbox{\rm ~ a.s.} \}$. 
Set  (cf. (\ref{martingale}))
\beq
\label{C^s.eq}
C^s (t):=\frac{C^o(t)}{C^o(s)} = e^{-\int_s^t \mu_C(r)dr} \cM_s (t),
\eeq
then the solution of (\ref{C.eq starting at s}) is 
$C^{s,y}(t;\nu) =C^s(t)\Big[y + \int_{[s,t)}\frac{f_C(u)}{C^s(u)}\,d\nu(u)\Big] =C^s(t)[y + {\overline\nu}(t)] $
where ${\overline\nu}(t) := \int_{[s,t)}\frac{f_C(u)}{C^s(u)}\,d\nu(u)$. 
The expected total discounted profit plus scrap value, net of investment, is 
\begin{eqnarray} \label{Js}
& & \hspace{-1cm}  
\cJ_{s,y}(\nu)= E\bigg\{\int_s^T  e^{-\int_s^t \mu_F (u) du} \ti R(C^{s,y}(t;\nu), w(t), r(t))\, dt 
        + e^{- \int_s^T \mu_F (u) du}G(C^{s,y}(T;\nu))  \nonumber\\ 
 &  & \hspace{1.5cm}  
        - \int_{[s, T)}  e^{- \int_s^t \mu_F (u) du} \,d\nu(t) \bigg\},   
\end{eqnarray}
and the firm's optimal capacity expansion problem is  
\beq\label{cap prob from s}
V(s, y) := \dts{\max_{\nu \in {\cal S}_s}}\ \cJ_{s,y}(\nu). 
\eeq 
Notice that, due to Assumption-[det], $\nu(t)$  and $C^s(t)$ are ${\cal F}_{s,t}$-measurable  where 
\[
{\cal F}_{s,t} := \sigma\{W(u)-W(s) : s\leq u \leq t \}. 
\]

The variational approach to the singular stochastic control problem (\ref{cap prob from s}) is based on the study of the optimal stopping problem naturally associated to it, 
 \beq\label{Z s}
Z^{s,y}(t):= \mbox{\rm ess}\hspace{-0.3cm}\inf_{\hspace{-0.3cm}\tau\in\U_s[t,T]} E\{\zeta^{s,y}(\tau)\,|\,{\cal F}_{s,t}\}, \qquad t\in [s,T],  
\eeq
where $\U_s [t,T]$ denotes the set of all $\{{\cal F}_{s,u}\}_{u\in [t,T]}$- stopping times taking values in $[t,T]$ (i.e. such that $\{\tau <u\} \in {\cal F}_{s,u}$), and 
\begin{eqnarray}\label{opport cost s}
\hspace{0.5cm}\zeta^{s,y}(t)&:=&
 \int_s^t e^{-\int_s^u \mu_F(r)\,dr}\, C^s(u) \tR_C(yC^s(u), w(u), r(u) )\,du \\
& &\hspace{0.3cm}+e^{-\int_s^t\mu_F(r)\,dr}\,\frac{C^s(t)}{f_C(t)} \ind_{\{t<T\}} + ~ e^{-\int_s^T\mu_F(r)\,dr}\,C^s(T)G'(yC^s(T)) \ind_{\{t=T\}}\nonumber
\end{eqnarray}
is the {\em opportunity cost of not investing until time} $t$ when the capacity is $y$ at time $s$.
Then the {\em optimal risk of not investing until time} $t$ is defined as 
\beq\label{v s}
v(s,y) := \cZ^{s,y}(s),
\eeq
where  $\cZ^{s,y}(\newdot)$ is a modification of $Z^{s,y}(\newdot)$ having right-continuous paths with left limits (``rcll"). 
Hence, up to a null set, $v(s,y) = Z^{s,y}(s)$. 

Consider the optimal stopping time
\beq\label{optimal tau s}
{\hat\tau}(s,y):=\inf\{t\in[s,T):\cZ^{s,y}(t)= \zeta^{s,y}(t)\}\wedge T,
\eeq
and take its left-continuous inverse (modulo a shift) 
\beq\label{bar nu s}
\left\{
\matrix{
\overline\nu^{s,y}(t):=\left[\sup\{z\geq y: {\hat\tau}(s,z+)<t\}-y\right]^+~\mathrm{~if~} t>s,\hfill\cr
\overline\nu^{s,y}(s) :=0. \hfill\cr
} 
\right.
\eeq
Notice that ${\hat\tau}(s,y)$ is non-decreasing in $y$ a.s. (cf. \cite{BK}, Lemma 1).
The following result follows by arguments as in the proof of  \cite{CHi}, Theorem~3.1, based on the  estimates in Proposition \ref{cJprop} and the strict concavity of $\cJ_{s,y}$ on $\cS_s$. 
\begin{proposition} \label{hnu s.thm}
Under Assumption-{\rm [det]}, for fixed $y>0$ and $s\in [0, T)$  set
\beq \label{hnu s}
\left\{
\matrix{
\hat{\nu}^{s,y}(t):= \int_{[s,t)}\frac{C^s(u)}{f_C(u)}\,d\overline \nu^{s,y}(u), \quad t\in(s,T),\hfill\cr
\hat{\nu}^{s,y}(s):= 0, \hfill\cr
}\right.
\eeq
then\\
$~~~~~$ {\rm (i)}~ $\hat{\nu}^{s,y}$  is the unique optimal solution of  $V(s,y) := \dts\max_{\nu\in\cS_s}  \cJ_{s,y}(\nu)$,
 
 {\rm (ii)}~ $ E\{\|\hat{\nu}^{s,y}\|_T \}\leq 2K_\cJ(1+y)\,{\dts\,\max_{t \in [s, T]}}\,e^{\int_s^t\mu_F(r)\,dr}$,

 {\rm (iii)}~ if $C(T;\hat{\nu}^{s,y}) \equiv 0$ a.s. then  $y=0$, ${\hat\nu}^{s,y}\equiv 0$; moreover
\begin{eqnarray}\label{Cns}
& & E\bigg\{ \int_t^T  e^{-\int_s^u \mu_F(r)\,dr}\tR_C(y C^s(u), w(u), r(u))\frac{C^s(u)}{C^s(t)}\,du \\
& & \hspace{2.4cm} +  e^{-\int_s^T \mu_F(r)\,dr}G'(y C^s(T)) \frac{C^s(T)}{C^s(t)}\,\bigg|\,\cF_{s,t}\bigg\} 
\leq  e^{-\int_s^t \mu_F(r)\,dr}\,\frac{1}{ f_C(t)} \,\mbox{~ a.e.}, \mbox{~a.s.} \nonumber
\end{eqnarray}
 \end{proposition}
\begin{remark}\label{Cobb-Douglas continued}\rm
\vspace{-0.3cm} 
If $R$ is of the Cobb-Douglas type (cf. Remark \ref{Appendix of [5]}), then $\tR_C(0, wt), r(t)) = +\infty$ for any $t$, so (\ref{Cns}) fails and $C^{y}(t;\hat\nu)>0$ for all $t>0$  (see also \cite{CHi}, Remark~3.2).
{\qed \par}\end{remark}
A generalization of \cite{BK}, Proposition 2 and Theorem 3, shows that the stopping time ${\hat\tau}(s,y)$ (cf. (\ref{optimal tau s}))  is optimal for $v(s,y)$ and the value function $v(s,y)$ is {\em the shadow value of installed capital}, i.e.
\vspace{-0.2cm}
\beq\label{v derivative V at s}
v(s,y) = \frac{\partial}{\partial y} V(s,y). \vspace{-0.2cm}
\eeq

Having defined our ingredients for all $s$ and $y$, observe that if ${\hat\tau}(s,y)=s$ for $s<T$,  then $\frac{1}{f_C(s)}=\zeta^{s,y}(s) =\cZ^{s,y}(s) \leq \cZ^{s,z}(s) \leq \zeta^{s,z}(s) = \frac{1}{f_C(s)}$ for all $z<y$, since $\tR_C$ and $G'$ are non-increasing.  Hence
${\hat\tau}(s,z)=s$ for all $z<y$ and the mapping
\beq
\label{freebdy} 
\hat y (s):=\sup\{z> 0\,:\,{\hat\tau}(s,z)=s\}
\eeq
 is well defined. In essence, $\hat y (s)$ is the maximal initial capacity at time $s$ for which it is optimal to invest instantaneously.

Define the probability measure $Q_s \sim P$ by 
\[
\frac{dQ_s}{dP}=\exp\left\{\int_s^T\sigma_C^\top (t) \,dW(t) -\frac{1}{2} \int_s^T \|\sigma_C (t)\|^2 du \right\}
= C^s (T)\, e^{\int_s^T \mu_C (t) dt},
\]
then $W^{Q_s} (t) :=  W(t) - W(s)- \int_s^t \sigma_C (u) du $ is a Wiener process. Under the new probability $Q_s$ the process $Y^{s,y}(t):=y\,C^s(t)$ evolves according to the dynamics
\bey
\label{newY s}
\left\{ \begin{array}{l}
dY^{s,y}(t)=Y^{s,y}(t)\left[(\|\sigma_C(t)\|^2 -\mu_C (t))\,dt + \sigma_C^\top (t) \,dW^{Q_s} (t) \right],\qquad t\in (s,T],\\
Y^{s,y}(s)=y,
\end{array}\right.
\eey
and the ``optimal risk of not investing'' becomes
\bey
\label{vv s}
& &  \hspace{-1cm} v(s,y) = \inf_{\tau\in\U_s[s,T]}E^{Q_s}\bigg\{\int_s^\tau e^{- \int_s^u \bar\mu (r)\, dr} 
\tR_C(Y^{s,y}(u),  w(u), r(u))\,du \\
& & \hspace{3cm} + e^{-\int_s^{\tau} \bar\mu (r) \, dr} \frac{1}{f_C(\tau)} \ind_{\{\tau<T\}}
 + e^{-\int_s^{\tau} \bar\mu (r) \, dr} G'(Y^{s,y}(\tau)) \ind_{\{\tau=T\}}\bigg\} \nonumber
\eey
with $\bar\mu (t) := \mu_C (t) + \mu_F (t)\geq\ve_o>0$ (as it is in Section~\ref{model}).
Taking $\tau=s$ in (\ref{vv s}) shows that $v(s, y)\leq\frac{1}{f_C(s)}$ for all $y>0$, $s\in[0,T)$. 
Hence the strip $[0,T) \times (0,\infty)$ splits into two regions (see also \cite{CHi}),
the {\em Continuation Region} (or inaction region) where it is not optimal to invest as the capital's replacement cost strictly exceeds the shadow value of installed capital, 
\beq\label{Delta}
\Delta = \Big\{(s,y) \in [0,T) \times (0,\infty) : v(s,y) < \frac{1}{f_C (s)} \Big\},  
\eeq
and its complement, the {\em Stopping Region} (or action region),
\beq\label{Delta complement}
\Delta^c = \Big\{(s,y) \in [0,T) \times (0,\infty) : v(s,y) = \frac{1}{f_C (s)} \Big\},  
\eeq
where it is optimal to invest instantaneously.

Under our reduced production function $\tR$, time-inhomogeneous capacity process and scrap value at terminal time $T$, the following result holds  and shows that the optimal stopping time ${\hat\tau}(s,y)$ of (\ref{vv s}) may be characterized as the first time $(t,Y^{s,y}(t))$ exits the Borel set $\Delta$.  The proof is omitted as it is a generalization of  \cite{CH1}, Theorem 4.1, established for the model with constant coefficients, discount factor and conversion factor $f_C$  but without scrap value,  later on extented to the model with scrap value and additive production function in \cite{CHi}, Proposition 3.3. Similar arguments were then used in \cite{ChF}, Theorem A.1,   in the case of time-dependent coefficients, discount factor and conversion factor $f_C$, but without scrap value. In all these results the derivative of the production function w.r.t. capacity $C$ was a function of $C$ only. 
\begin{theorem}
\label{stop} 
Under Assumption-{\rm [det]}, 
the optimal stopping time of problem $(\ref{vv s})$ is characterized as 
\beq
\label{tau*bdy}
{\hat\tau}(s,y)=\inf\Big\{t\in[s,T):v(t,Y^{s,y}(t))= \frac{1}{f_C (t)}\Big\}\wedge T
\eeq
for each starting time $s\in[0,T)$ and initial state $y\in(0,\infty)$. 
\end{theorem}
It follows that (\ref{freebdy})  is now written as 
\beq\label{hat y in function of v}
\hat y (s) =\sup \Big\{z\geq 0:  v(s, z)= \frac{1}{f_C (s)}\Big\},
\eeq
hence it is the lower boundary of the Borel set $\Delta$ in the $(s,y)$-plane; (\ref{tau*bdy}) is equivalently written as 
\beq\label{tau by hat y}
\hat\tau (s,y) =\inf\{t\in [s,T) :  [Y^{s,y}(t) -  \hat y (t)]^+ = 0 \} \wedge T
\eeq
and so, recalling $Y^{s,y}(t)=y C^s(t)$, its left-continuous inverse (\ref{bar nu s}) may be written in the form 
\beq\label{bar nu by hat y}
\left\{
\matrix{
\overline\nu^{s,y}(t):=\dts\sup_{s\leq u<t}\Big[\frac{\hat y(u)}{C^s(u)} - y \Big]^+~~\mathrm{~for~} t>s,\hfill\cr
\overline\nu^{s,y}(s) :=0. \hfill\cr
} 
\right.
\eeq
Hence $y+ \overline\nu^{s,y}(t) \geq \frac{\hat y(t)}{C^s(t)}$ a.s. (i.e. $C^{s,y}(t;\hat{\nu}^{s,y}) \geq \hat y(t)$ a.s.), and $\overline\nu^{s,y}(t) = \frac{\hat y(t) - y C^s(t)}{C^s(t)}$  at times of strict increase.


\end{document}